\documentclass[final,preprint,3p,sort&compress]{elsarticle}
\journal{{\tt arXiv.org}}
\usepackage[dvips]{epsfig}
\usepackage{graphicx} 
\usepackage{pdfpages}
\usepackage{latexsym}
\usepackage{verbatim}
\usepackage{doi}        
\usepackage{amsmath,amsthm}
\usepackage{amssymb}
\usepackage{bbm}
\usepackage{enumerate}
\usepackage{placeins}
\usepackage{standalone}
\usepackage{paralist}
\usepackage{subcaption}
\usepackage{booktabs} % Required for nicer horizontal rules in tables
\usepackage{standalone}
\definecolor{greenyellow}   {cmyk}{0.15, 0   , 0.69, 0   }
\definecolor{yellow}        {cmyk}{0   , 0   , 1   , 0   }
\definecolor{goldenrod}     {cmyk}{0   , 0.10, 0.84, 0   }
\definecolor{dandelion}     {cmyk}{0   , 0.29, 0.84, 0   }
\definecolor{apricot}       {cmyk}{0   , 0.32, 0.52, 0   }
\definecolor{peach}         {cmyk}{0   , 0.50, 0.70, 0   }
\definecolor{melon}         {cmyk}{0   , 0.46, 0.50, 0   }
\definecolor{yelloworange}  {cmyk}{0   , 0.42, 1   , 0   }
\definecolor{orange}        {cmyk}{0   , 0.61, 0.87, 0   }
\definecolor{burntorange}   {cmyk}{0   , 0.51, 1   , 0   }
\definecolor{bittersweet}   {cmyk}{0   , 0.75, 1   , 0.24}
\definecolor{redorange}     {cmyk}{0   , 0.77, 0.87, 0   }
\definecolor{mahogany}      {cmyk}{0   , 0.85, 0.87, 0.35}
\definecolor{maroon}        {cmyk}{0   , 0.87, 0.68, 0.32}
\definecolor{brickred}      {cmyk}{0   , 0.89, 0.94, 0.28}
\definecolor{red}           {cmyk}{0   , 1   , 1   , 0   }
\definecolor{orangered}     {cmyk}{0   , 1   , 0.50, 0   }
\definecolor{rubinered}     {cmyk}{0   , 1   , 0.13, 0   }
\definecolor{wildstrawberry}{cmyk}{0   , 0.96, 0.39, 0   }
\definecolor{salmon}        {cmyk}{0   , 0.53, 0.38, 0   }
\definecolor{carnationpink} {cmyk}{0   , 0.63, 0   , 0   }
\definecolor{magenta}       {cmyk}{0   , 1   , 0   , 0   }
\definecolor{violetred}     {cmyk}{0   , 0.81, 0   , 0   }
\definecolor{rhodamine}     {cmyk}{0   , 0.82, 0   , 0   }
\definecolor{mulberry}      {cmyk}{0.34, 0.90, 0   , 0.02}
\definecolor{redviolet}     {cmyk}{0.07, 0.90, 0   , 0.34}
\definecolor{fuchsia}       {cmyk}{0.47, 0.91, 0   , 0.08}
\definecolor{lavender}      {cmyk}{0   , 0.48, 0   , 0   }
\definecolor{thistle}       {cmyk}{0.12, 0.59, 0   , 0   }
\definecolor{orchid}        {cmyk}{0.32, 0.64, 0   , 0   }
\definecolor{darkorchid}    {cmyk}{0.40, 0.80, 0.20, 0   }
\definecolor{purple}        {cmyk}{0.45, 0.86, 0   , 0   }
\definecolor{plum}          {cmyk}{0.50, 1   , 0   , 0   }
\definecolor{violet}        {cmyk}{0.79, 0.88, 0   , 0   }
\definecolor{royalpurple}   {cmyk}{0.75, 0.90, 0   , 0   }
\definecolor{blueviolet}    {cmyk}{0.86, 0.91, 0   , 0.04}
\definecolor{periwinkle}    {cmyk}{0.57, 0.55, 0   , 0   }
\definecolor{cadetblue}     {cmyk}{0.62, 0.57, 0.23, 0   }
\definecolor{cornflowerblue}{cmyk}{0.65, 0.13, 0   , 0   }
\definecolor{midnightblue}  {cmyk}{0.98, 0.13, 0   , 0.43}
\definecolor{navyblue}      {cmyk}{0.94, 0.54, 0   , 0   }
\definecolor{royalblue}     {cmyk}{1   , 0.50, 0   , 0   }
\definecolor{blue}          {cmyk}{1   , 1   , 0   , 0   }
\definecolor{cerulean}      {cmyk}{0.94, 0.11, 0   , 0   }
\definecolor{cyan}          {cmyk}{1   , 0   , 0   , 0   }
\definecolor{processblue}   {cmyk}{0.96, 0   , 0   , 0   }
\definecolor{skyblue}       {cmyk}{0.62, 0   , 0.12, 0   }
\definecolor{turquoise}     {cmyk}{0.85, 0   , 0.20, 0   }
\definecolor{tealblue}      {cmyk}{0.86, 0   , 0.34, 0.02}
\definecolor{aquamarine}    {cmyk}{0.82, 0   , 0.30, 0   }
\definecolor{bluegreen}     {cmyk}{0.85, 0   , 0.33, 0   }
\definecolor{emerald}       {cmyk}{1   , 0   , 0.50, 0   }
\definecolor{junglegreen}   {cmyk}{0.99, 0   , 0.52, 0   }
\definecolor{seagreen}      {cmyk}{0.69, 0   , 0.50, 0   }
\definecolor{green}         {cmyk}{1   , 0   , 1   , 0   }
\definecolor{forestgreen}   {cmyk}{0.91, 0   , 0.88, 0.12}
\definecolor{pinegreen}     {cmyk}{0.92, 0   , 0.59, 0.25}
\definecolor{limegreen}     {cmyk}{0.50, 0   , 1   , 0   }
\definecolor{yellowgreen}   {cmyk}{0.44, 0   , 0.74, 0   }
\definecolor{springgreen}   {cmyk}{0.26, 0   , 0.76, 0   }
\definecolor{olivegreen}    {cmyk}{0.64, 0   , 0.95, 0.40}
\definecolor{rawsienna}     {cmyk}{0   , 0.72, 1   , 0.45}
\definecolor{sepia}         {cmyk}{0   , 0.83, 1   , 0.70}
\definecolor{brown}         {cmyk}{0   , 0.81, 1   , 0.60}
\definecolor{tan}           {cmyk}{0.14, 0.42, 0.56, 0   }
\definecolor{gray}          {cmyk}{0   , 0   , 0   , 0.50}
\definecolor{black}         {cmyk}{0   , 0   , 0   , 1   }
\definecolor{white}         {cmyk}{0   , 0   , 0   , 0   } 

 \usepackage[]{hyperref} 

\newcommand{\externaltikz}[2]{\includegraphics{Externals/t#1}}
\usepackage{tikz,pgfplots}
\pgfplotsset{compat=1.12}
\usepgfplotslibrary{groupplots}

\newlength\figurewidth
\newlength\figureheight
\usepackage{relinput}

\newcounter{tikzsubfigcounter}[figure]
\renewcommand{\thetikzsubfigcounter}{\thesection.\the\numexpr\value{figure}+1\relax\alph{tikzsubfigcounter}}

\newcounter{tikzsubfigcounterinvisible}[figure]
\renewcommand{\thetikzsubfigcounterinvisible}{\thesection.\the\numexpr\value{figure}+1\relax\alph{tikzsubfigcounterinvisible}}

\newcommand{\settikzlabel}[1]{ %
\refstepcounter{tikzsubfigcounterinvisible} \label{#1} 
}

\newtheorem{thm}{\bf Theorem}[section]
\newtheorem{lem}[thm]{\bf Lemma}

\numberwithin{equation}{section}
\newcommand{\bdm}{\begin{displaymath}}
\newcommand{\edm}{\end{displaymath}}
\newcommand{\beq}{\begin{equation}}
\newcommand{\eeq}{\end{equation}}
\newcommand{\beqa}{\begin{eqnarray}}
\newcommand{\eeqa}{\end{eqnarray}}
\title{A Comparative Study of Limiting Strategies in Discontinuous Galerkin Schemes for the $M_1$ Model of Radiation Transport}
\author[pc]{Prince Chidyagwai}
\author[mf]{Martin Frank}
\author[fs]{Florian Schneider}
\author[bs]{Benjamin Seibold}
\address[pc]{Department of Mathematics and Statistics, Loyola University Maryland, 4501 N Charles Street, Baltimore, MD 21210 {\tt pchidyagwai@loyola.edu}}
\address[mf]{Department of Mathematics, RWTH Aachen University, Schinkelstr. 2,52062 Aachen, Germany, {\tt frank@mathcces.rwth-aachen.de}}
\address[fs]{Fachbereich Mathematik, TU Kaiserslautern, Erwin-Schr\"odinger-Str., 67663 Kaiserslautern, Germany, {\tt schneider@mathematik.uni-kl.de}}
\address[bs]{Department of Mathematics, Temple University, 1805 N Broad Street, Philadelphia, PA 19122, {\tt seibold@temple.edu}}
\date{}

\newcommand{\abs}[1]{\left|#1\right|}

\newcommand{\bfU} { \mbox{\boldmath $U$} }

\newcommand{\bfn} { \mbox{\boldmath $n$} }
\newcommand{\bfF} { \mbox{\boldmath $F$} }
\newcommand{\bfM} { \mbox{\boldmath $M$} }
\newcommand{\bfH} { \mbox{\boldmath $H$} }
\newcommand{\bfS} { \mbox{\boldmath $S$} }
\newcommand{\sF}  { \mbox{$\mathcal{F}$} }

\newcommand{\psiz}{\psi^{(0)}}
\newcommand{\psio}{\psi^{(1)}}
\newcommand{\psit}{\psi^{(2)}}
\newcommand{\psizbar}{\bar{\psi}^{(0)}}
\newcommand{\psiobar}{\bar{\psi}^{(1)}}

\pgfplotsset{
plotstylea/.style={royalblue!70,every mark/.append style={solid,line width = 0pt,fill=royalblue!60!black},mark=ball}, plotstyleb/.style={dashed,no marks}, }

\begin{document}

\begin{abstract}
    The $M_1$ minimum entropy moment system is a system of hyperbolic balance laws that approximates the radiation transport equation, 
    and has many desirable properties. 
    Among them are symmetric hyperbolicity, entropy decay, moment realizability, and correct behavior in
    the diffusion and free-streaming limits. However, numerical difficulties arise when approximating the solution
    of the $M_1$ model by high order numerical schemes; namely maintaining the realizability of the numerical solution
    and controlling spurious oscillations. In this paper, we extend a previously constructed one-dimensional realizability limiting strategy to 2D. 
    In addition, we perform a numerical study of various combinations of the realizability limiter and the TVBM local slope limiter 
    on a third order Discontinuous Galerkin (DG) scheme on both triangular and rectangular meshes. In several test cases, 
    we demonstrate that in general, a combination of the realizability limiter and a TVBM limiter is necessary to obtain a robust and accurate numerical scheme.
    Our code is published so that all results can be reproduced by the reader.
\end{abstract}
\maketitle

\noindent

\section{Introduction}
The $M_1$ model of radiative transfer is a nonlinear system of hyperbolic balance laws, and reads
\begin{subequations}
\begin{align}
    \partial_t \psiz + \nabla_x\cdot \psio &= -\sigma_a \psiz + q^{(0)} \label{eq:m1_1} \\
    \partial_t \psio + \nabla_x \cdot \psit(\psiz,\psio) &= -(\sigma_s+\sigma_a) \psio + q^{(1)}.\label{eq:m1_2}
\end{align}
\end{subequations}

The quantities $\psiz$, $\psio$, $\psit$ are, respectively, the zeroth (particle density), first (mean velocity) and second moment (pressure) over the unit sphere $S^2$ of the angular flux $\psi$, 
$$
\psiz := \int_{S^2} \psi(\Omega)d\Omega,\quad \psio := \int_{S^2} \Omega\psi(\Omega)d\Omega,\quad\psit := \int_{S^2} \Omega\Omega^T\psi(\Omega)d\Omega, 
$$
and are, respectively, a scalar, a vector, and a matrix. The system is supplemented by the closure condition
\begin{equation}
\label{eq:closure}
\psit(\psiz,\psio) = D(\frac{\psio}{\psiz}) \psiz,
\end{equation}
where 
\begin{gather}
D(n) = \frac{1-\chi(|n|)}{2}\text{id} + \frac{3\chi(|n|)-1}{2}\frac{nn^T}{|n|^2}\text{ and }\chi(f) = \frac{3+4f^2}{5+2\sqrt{4-3f^2}}~\text{ for } f\in[0,1].
\label{eq:D_exp}
\end{gather}
Here, $D$ is a matrix, $n=\frac{\psio}{\psiz}$ is a vector, and $f=|n|$ is a scalar. The quantity $\chi(f)$ is called the Eddington factor \cite{Lev84,BruHol01}. Let $X \subset \mathbb{R}^2$ be a bounded polygonal domain, then the system \eqref{eq:m1_1} -\eqref{eq:m1_2} 
can be written as a general first-order system of balance laws
\begin{gather}
	\dfrac{\partial U}{\partial t} + \nabla \cdot \sF   =  S,\text{ in } X \times (0,T) \label{eq:standard_form}\\
U(x,y,0) = u_{0}(x,y),\text{ for } (x,y)  \in X \label{eq:initial_condition}\\
\end{gather}
where $\displaystyle \sF = [F,G]$ and
$$
U = \begin{bmatrix} \psiz \\ \psio_x \\ \psio_y \end{bmatrix},\ 
F = \begin{bmatrix} \psio_x \\ \psit_{xx} \\ \psit_{xy} \end{bmatrix},\ 
G = \begin{bmatrix} \psio_y \\ \psit_{yx} \\ \psit_{yy} \end{bmatrix},\ 
S = \begin{bmatrix} -\sigma_a \psiz + q^{(0)} \\ -(\sigma_s+\sigma_a) \psio_x + q^{(1)}_x \\ -(\sigma_s+\sigma_a) \psio_y + q^{(1)}_y \end{bmatrix},
$$
keeping in mind the closure relation \eqref{eq:closure}. Here, $\psio_x$ and $\psio_y$ denote the first and second component of $\psio$, respectively, similarly for $\psit_{xx}$ to $\psit_{yy}$. 

The system has to be complemented with boundary conditions of the form
\begin{align}
U(x,y,t) = \gamma(x,y,t),\text{ for } (x,y) \in I(\partial X,U) \times (0,T) \label{eq:boundary_condition},
\end{align}
where the operator $I$ returns the influx boundaries, i.e.\  those parts of the boundary where information is transported into the domain \cite{toro2009riemann}. Whether a part of the boundary is an influx boundary also depends on the solution itself, since the information direction can be read off the sign of the eigenvalues of the directional Jacobian in normal direction at these points. In our numerical experiments, we consider
compactly supported initial data which do not reach the boundary
before the final time. We prescribe Dirichlet boundary conditions
identical to the initial conditions, hence the numerical solution is
unaffected by the boundaries.

The expression for $D$ in \eqref{eq:D_exp} comes from closing the moment system by an entropy closure using the entropy for photons. See for example \cite{Dre87,DubFeu99,Frank07,jaynes1957info,Lev84,levermore1996moment,MullerRuggeri,Ore55,Ros54} and references therein for more information. These references also discuss many properties which make the $M_1$ model, and entropy closures in general, quite appealing. Among these are symmetric hyperbolicity, i.e.\  the system can be transformed into a symmetric hyperbolic system, and a natural entropy-entropy flux pair, both of which ensure some level of well-posedness \cite{toro2009riemann,LeVeque2002}. The hyperbolicity is of main interest in this work, and will be discussed in detail below. We especially focus on its connection to realizability -- the fact that a moment vector can be reproduced by a non-negative particle density.

When the $M_1$ model is discretized using a monotone first-order scheme for hyperbolic equations (e.g. the Lax-Friedrichs scheme), it can be shown that the numerical solution will never leave the set of realizable moments \cite{Olbrant12}. This means that starting with a realizable initial condition, the numerical solution will be realizable at every time step. This is in particular required, since the model is not well-defined outside of the realizability set.

Unfortunately, and this is the main topic of this paper, higher-order numerical schemes do not automatically preserve this property, as will be discussed later. We use the Discontinuous Galerkin (DG) method, which provides a general framework to construct numerical schemes of arbitrary approximation order to solve hyperbolic balance laws. The original DG method was introduced in $1973$ by Reed and Hill \cite{reedhill} for neutron transport and has been developed further to the Runge-Kutta Discontinuous Galerkin (RKDG) method by Cockburn et al.\ in a series of papers \cite{CockburnShuIV,CockburnShuPk,CockburnShuP1,CockburnLinShu,CockburnShuV}. A $(k+1)st$ order RKDG method uses a piecewise-polynomial approximation of degree $k$ in space and a ($k+1)st$ order strong stability preserving explicit Runge-Kutta scheme in time. 

A DG scheme has to be supplemented by a limiting strategy. This typically consists of a slope limiter which ensures stability of the solution. A standard example of such a limiter is the TVBM limiter \cite{CockburnLinShu}. For a system of balance laws, limiting the
conserved variables component-wise can result in oscillations, due to the Gibbs phenomenon \cite{Shu1998}. Instead, the limiter has to be applied in characteristic variables, i.e.\ \ Riemann invariants that are obtained by diagonalizing the system Jacobian \cite{biswas1994parallel,CockburnShuV}. In general, limiting in the characteristic variables gives superior results \cite{review_article,Olbrant12}. However, there are numerous examples
for which the numerical solution obtained by limiting in the characteristic variables and component-wise limiting are comparable, e.g. \cite{biswas1994parallel,li2011central,liu2007runge,liangdiscontinuous,liu_xu_gas_kinetic}. In these situations, since the transformation to and from the characteristics requires additional computational effort, component-wise 
limiting is faster. 

It has been shown in \cite{Olbrant12} that a robust numerical approximation of the $M_1$
model in 1D requires an additional limiter: a realizability limiter
which ensures that the model remains well-posed. This has been extended to the general, one-dimensional, $M_N$ case in \cite{Schneider2015a}. In the current work, we extend this realizability limiting strategy to two dimensions. In addition, we demonstrate that in the absence of the realizability limiter, the TVBM limiter applied in the characteristic variables
leads to a qualitatively superior solution compared to component-wise limiting; however, this limiter alone is not enough to guarantee moment realizability of the scheme.

The $M_1$ model therefore might serve as a good benchmark test for DG implementations, because it absolutely requires very careful limiting. 
We arrive at this conclusion experimentally, by running several well-known radiation transport test cases with two independently developed third-order 
DG schemes on unstructured triangular grids and regular rectangle grids. 

The rest of this paper is organized as follows. 
In Section~\ref{sec:M1}, we describe in detail the concept of realizability, and how it is connected to the well-posedness of the $M_1$ model. 
The DG implementations are described in Section~\ref{sec:DG}. The realizability limiter will be developed and analyzed in 
Section \ref{sec:limiter}. Section \ref{sec:numerical_results} contains the results from the test cases.

\section{Properties of the $M_1$ Model}
\label{sec:M1}
From the general theory of moments one can deduce
conditions on a given set of values, that are necessary and sufficient
for the existence of a probability measure whose moments match these
values \cite{Ker76}. Such moments are called realizable. In our case, we can restrict ourselves to measures that have a formal density (i.e.\ for the closure we allow Dirac $\delta$'s as densities).
Given a scalar $\psiz$ and a vector $\psio$, these are the zeroth and first moment of a non-negative density if and only if \cite{Ker76,Schneider2016,Monreal}
$$
\psiz>0 \text{ and } |\psio|\leq \psiz.
$$
The interior of the realizable set is described by $\psiz>0 \text{ and } |\psio| < \psiz$.
It has been shown in \cite{Goudon2005} that, under reasonable assumptions on the initial conditions, the analytical solution of the $M_1$ model in one spatial dimension remains realizable for all time. Until now, no similar statement could be made for the $M_1$ model in multiple dimensions or higher-order $M_N$ models. 

As mentioned before, the hyperbolicity of the $M_1$ model is a direct consequence of the entropy closure. Additional insight can be obtained by computing the eigenvalues explicitly. Given a unit direction vector $\bfn=(n_x,n_y)^T$, we have to determine the eigenvalues of the directional Jacobian
\begin{equation}
\label{eq:jacobian}
n_xJ_x(U) + n_yJ_y(U),
\end{equation}
where $J_x$ and $J_y$ are the Jacobians of $F$ and $G$, respectively. Because of the invariance of the model under coordinate transformations, the eigenvalues can only depend on the angle between $\psio$ and $\bfn$, and the absolute value $\frac{|\psio|}{\psiz}$. The eigenvalues can be computed analytically, but the formulas are very lengthy so we do not show them here (cf.\ \cite{Berthon2006}). Figure \ref{fig:EV} shows the eigenvalues for two different angles. The most important observation is that the eigenvalues collapse into one value at the boundary of the realizability region, i.e.\ for $\frac{|\psio|}{\psiz}=1$. Inspection of the Jacobian shows that at that boundary the Jacobian is no longer diagonalizable, i.e.\ the $M_1$ model loses (strict) hyperbolicity. This means that the $M_1$ model is only well-posed in the interior of the domain of realizability. As a consequence,  in a numerical scheme one should always ensure staying in the interior of the relizability domain.

\begin{figure}

\centering
\externaltikz{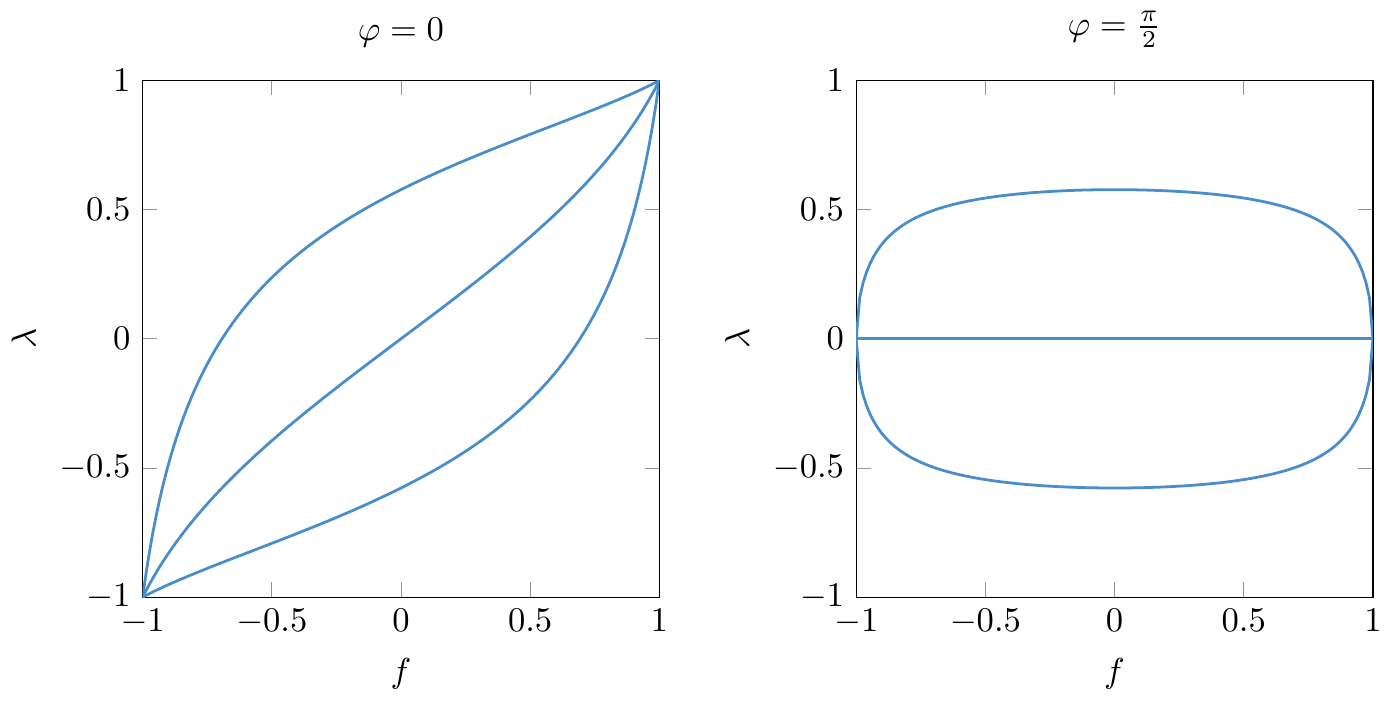}{\input{Images/M1eigenvalues1}}  
\caption{Eigenvalues of the $M_1$ model as a function of $f=\frac{n\cdot\psio}{\psiz}$ for a fixed angle $\varphi$ between $\psio$ and $n$.}
\label{fig:EV}
\end{figure}
\section{Runge-Kutta Discontinuous Galerkin Method}
\label{sec:DG}

\subsection{Spatial Discretization}
In this section, we describe the Runge-Kutta Discontinuous Galerkin method (RKDG) for solving~\eqref{eq:standard_form}--\eqref{eq:boundary_condition}. 
Following the approach outlined in a series of papers by Cockburn and Shu
\cite{CockburnShuIV,CockburnShuPk,CockburnShuP1,CockburnShuV}, we discretize in space using piece-wise
polynomials of degree $k$, that are allowed to be discontinuous at the cell interface. The time discretization is performed by a strong stability preserving explicit Runge-Kutta scheme of order $k+1$ \cite{gottlieb_shu_tadmor}.

In the following, let $\mathcal{T}_h$ be a partition of a polygonal computational domain $X$ and $K$ be
an element in $\mathcal{T}_h$ with boundary edges $e$. For each $t \in [0,T]$, we seek an approximate solution 
$U_h(x,t)$ in the DG space 
\[
V_h^k = \{v \in L^\infty(X): v|_{K} \in P^k(K), \forall K \in \mathcal{T}_h \}
\]
where $P^{k}(K)$ is the set of polynomials of degree $k$. We follow the Galerkin approach: first we multiply (\ref{eq:standard_form}) 
by a smooth test function $v$ and integrate by parts over each element. We replace the exact solution $U$ and smooth test function $v$ 
by the approximation solution $U_h$ and $v_h$ (both in $V_h^k$) respectively to obtain:
\begin{gather}
\label{eq:dweakform1}
\dfrac{d}{dt} \int_{K} U_h(x,t)v_h(x)\,dx + \sum_{e \in \partial K} \int_{e}\sF(U_h(x,t))\cdot
\bfn_{e,K} v_h(x)\,d\Gamma \nonumber\\ -\int_{K} \sF(U_h(x,t))	 \cdot \nabla v_h(x)\,dx 
= \int_{K} S\big(U_h(x,t) \big)v_h(x)\,dx \quad\forall v_h \in V_h^k,
 \label{eq:dweakform1a}
 \end{gather}
 where $\bfn_{e,K}$ is the outward unit normal to the boundary of $K$. We take $U_h(x,0)$ on each element to be the
 $L_2$-projection of the initial condition on $V_h^k$, i.e., 
 \begin{gather}
	 \int_{K} U_h(x,0)v_h(x)\,dx = \int_{K}u_0 v_h(x)\,dx, \quad \forall v_h \in V^k_h.
	 \label{eq:initial_condition}
 \end{gather}
The choice of a discontinuous basis implies that our approximate solution $U_h(x,t)$ is discontinuous across edges. 
In this case the normal trace $\displaystyle \sF(U_h(x,t))\cdot \bfn_{e,K}$ is not defined on the element boundary $\partial K$. 
We replace the normal trace by a numerical flux function $H_{e,K}(x,t)$ that depends
on the approximate solution from the interior and exterior of the element $K$. Thus we define
\begin{align}
\displaystyle H_{e,K}(x,t) = H_{e,K}(U_h(x^{int\{K\}},t), U_h(x^{ext\{K\}},t)), 
\label{eq:numerical_flux}
\end{align}
where
\begin{align*}
	U_h(x^{int\{K\}},t) &= \lim_{\xi\rightarrow (x,y) \in K}U_h(\xi,t),
\end{align*}
for the approximate solution defined from the interior of the element $K$ and
\begin{align*}
U_h(x^{ext\{K\}},t) = \left \{ \begin{array}{l} \gamma_h(x,t) \quad \textrm{ if } x \in \partial X,\\
                    \displaystyle\lim_{\xi\rightarrow (x,y) \notin K } U_h(\xi,t), \quad\textrm{otherwise} 
                    \end{array}
                    \right.
\end{align*}
for the points on the exterior of $K$.
Using the numerical flux \eqref{eq:numerical_flux}, the discrete weak formulation (\ref{eq:dweakform1})--(\ref{eq:initial_condition}) becomes
\begin{align}
\label{eq:dweakform2}
\dfrac{d}{dt} \int_{K} U_h(t,x)v_h(x)\,dx &+ \sum_{e \in \partial K}\int_{e} H_{e,K}(U_h(x^{int\{K\}},t),U_h(x^{ext\{K\}},t)) v_h(x)\,d\Gamma \nonumber\\
-\int_{K} \sF(U_h(x,t))\cdot \nabla v_h(x)\,dx                                                   
&= \int_{K} S(U_h(t,x))v_h(x)\,dx, \quad \forall v_h \in V^k_h,\\
\label{eq:dweakform2a}
\int_{K} U_h(x,0)v_h(x)\,dx &= \int_{K}u_0 v_h(x)\,dx, \quad \forall v_h \in V^k_h.
\end{align}
We choose the global Lax-Friedrichs flux
\begin{align}
\label{eq:LFFlux}
	H_{e,K}(a,b) = \dfrac{1}{2} \bigg[\sF(a)\cdot \bfn_{e,K} + \sF(b)\cdot \bfn_{e,K} - \alpha (b-a)\bigg],
\end{align}
where the numerical viscosity constant $\alpha$ is taken as the global estimate of the absolute value of the largest eigenvalue of the Jacobian \eqref{eq:jacobian}. 
For the $M_1$ model, we can take $\alpha=1$.

Boundary conditions \eqref{eq:boundary_condition} have to be incorporated via the quantities $U_h(x^{ext\{K\}},t)$ in the edge integral, if the edge $e$ is part of the boundary. We take the simplest approach possible by adding so-called ``ghost cells'' which are then filled with the corresponding values for $U$. 
Note, however, that the validity of this approach, due to its inconsistency with 
the original boundary conditions  \eqref{eq:boundary_condition}, is not entirely non-controversial, but the question of appropriate boundary conditions for
moment models is an open problem \cite{pomraning1964variational,Larsen1991,Rulko1991,Struchtrup2000,levermore2009boundary} which is not explored here. 

Equations (\ref{eq:dweakform2})--(\ref{eq:dweakform2a}) can be written as a system of ODEs on each element after inverting the mass matrix in~\eqref{eq:dweakform2}.
Indeed, let $\displaystyle \{\varphi_1, \varphi_2, \cdots ,\varphi_{N_k} \}$ denote a basis of the space of 
polynomials of degree $k$ on cell $K$, where $\displaystyle N_k = \dfrac{1}{2}(k+1)(k+2)$. On each element, the
DG approximate solution $U_h$ of the components has the form
\begin{align}
\label{eqn:solution_form}
\psi^{(0)} = \sum_{i=1}^{N_k}\alpha^{\psi^{(0)}}_i \varphi_i,\quad \psi_x^{(1)} = \sum_{i=1}^{N_k}\alpha^{\psi^{(1)}_x}_i \varphi_i, \textrm{ and }
\psi_y^{(1)} = \sum_{i=1}^{N_k}\alpha^{\psi^{(1)}_y}_i \varphi_i,
\end{align}
where $\alpha^{\psi^{(0)}}_i, \alpha^{\psi^{(1)}}_x, \alpha^{\psi^{(1)}}_y$ are unknowns to be determined.
Using the form of the approximate solution in (\ref{eqn:solution_form}), we can write (\ref{eq:dweakform2})-(\ref{eq:dweakform2a}) in matrix form:
\begin{gather}
\label{eqn:weakform_m}
\bfM\dfrac{\partial}{\partial t} \alpha_h + \bfH \alpha_h - \bfF\alpha_h = \bfS \alpha_h \\
\label{eqn:initial_condition}
\bfM \alpha^0_h = \bfU_0
\end{gather}
where $\alpha_h$ is the vector of solution coefficients and 
\begin{gather}
	(\bfM)_{ij} = \int_{K}\varphi_j \varphi_i,\quad (\bfH)_{i} = \sum_{e \in \partial K} \int_{e}H_{e,K}(x,t)\varphi_i,
    \quad (\bfF)_{i} = \int_{K}\mathcal{F}(U_h)\cdot \nabla \varphi_i\\
	(\bfS)_i = \int_{K} S(U_h)\varphi_i, \textrm{ and } (\bfU_0)_i = \int_{K}u_0\varphi_i.
\end{gather}
The complete coefficient vector is given by $\displaystyle \alpha_h = \{\alpha^{\psiz}_1, \cdots, \alpha^{\psiz}_{N_k}, \alpha^{\psio_x}_{1},\cdots,\alpha^{\psio_x}_{N_k},\alpha^{\psio_y}_{1},\cdots,\alpha^{\psio_y}_{N_k} \}$, where $\displaystyle \alpha^{\psiz}_i, \alpha^{\psio_x}_i, \alpha^{\psio_y}_i$ 
are the coefficients of the numerical approximation to the zeroth and first order moments, respectively. 
We can write \eqref{eqn:weakform_m}--\eqref{eqn:initial_condition} in the form
\begin{gather}
\label{eq:ode1}
\dfrac{d}{dt} \alpha_h = \mathcal{L}_h(\alpha_h), \quad \textrm{on } X \times (0,T) \\
\label{eq:ode2}
\alpha_h(x,0) = \alpha^{0}_h 
\end{gather}
where
\begin{gather}
\label{eq:ode3}
\mathcal{L}_h(\alpha_h) = \bfM^{-1}(\bfS\alpha_h + \bfF \alpha_h - \bfH \alpha_h).
\end{gather}
We approximate the solution by discontinuous quadratic polynomials in space and the third-order strong stability preserving Runge-Kutta time discretization scheme 
proposed in \cite{CWShu,Shu1988}, also known as the Shu-Osher scheme. Let $\{t^n\}_{n=0}^{N}$ be a partition of $[0,T]$ and let $\Delta t = t^{n+1} - t^n$, $n = 0, \cdots, N-1$, then the time
stepping scheme updating the coefficients of the DG polynomials can be written as:
\begin{center}
	\begin{itemize}
        \item  Set  $\alpha^0 =\bfM^{-1}\bfU_0$;
			\item  For  $n = 1, \cdots, N-1$  compute  $\alpha_h^{n+1}$ as follows:
				\begin{enumerate}
					\item $\displaystyle \alpha_h^{(1)} = \alpha_h^n + \Delta t_{n} \mathcal{L}_h(\alpha^n_h)$
					\item $\displaystyle \alpha_h^{(2)} = \dfrac{3}{4}\alpha_h^n + \dfrac{1}{4}(\alpha_h^{(1)} + \Delta t_{n} \mathcal{L}_h(\alpha_h^{(1)}))$
					\item $\displaystyle \alpha_h^{(3)}= \dfrac{1}{3}\alpha_h^{n} + \dfrac{2}{3}(\alpha_h^{(2)} + \Delta t_{n} \mathcal{L}_h(\alpha_h^{(2)}))$
					\item Set $\alpha_h^{n+1} = \alpha_h^{(3)}$,
				\end{enumerate}
		\end{itemize}
\end{center}
This Runge-Kutta method is a convex combination of (iterated) forward Euler steps. Using the convexity of the realizability domain, one can show 
that it preserves realizability under a specific CFL-condition. We use this property to achieve high order also in time without the need of dealing with complicated time-discretizations in the proof of realizability preservation.

\subsection{Quadrature Rules}
The assembly of the discrete operator $\displaystyle \mathcal{L}_h(\alpha_h)$ (\ref{eq:ode1}) is done using numerical
quadrature that is exact for polynomials of degree $2k+1$ for edge integrals, and degree $2k$ for volume
integrals, respectively for both the rectangular and triangular meshes. In addition, we will need a modified Gaussian quadrature 
rule to construct the realizability limiter. This rule consists of quadrature points in the interior of the cell and on its boundary where the latter form one-dimensional quadratures on the element 
edges. This is crucial to balance the different appearing types of spatial integrals in the proof of the realizability-preserving property.
These rules are described in the following section that closely follows the construction in \cite{ZhangShu12}.

\subsubsection{Triangles}
We start with two quadrature rules on the interval $[-\frac{1}{2},\frac{1}{2}]$, a Gaussian rule
$$
v^\beta = \frac{1}{2}\begin{pmatrix}
-\sqrt{\frac{3}{5}}&0&\sqrt{\frac{3}{5}}
\end{pmatrix},\quad w^\beta = \frac{1}{18} \begin{pmatrix}
5 & 8 & 5
\end{pmatrix}
$$
and a Gauss-Lobatto rule
\begin{align}
\label{eq:GLrule}
u^\alpha = \frac{1}{2}\begin{pmatrix}
-1&0&1
\end{pmatrix},\quad
\hat{w}^\alpha = \frac{1}{6} \begin{pmatrix}
1 & 4 & 1
\end{pmatrix}.
\end{align}
We map the tensor product rule of those Gauss and Gauss-Lobatto quadrature rules from the square $[-\frac{1}{2},\frac{1}{2}]^2$ onto the triangle $K$ with vertices $V_1^K,\ldots,V_3^K$ using the three degenerate projections 
\begin{align*}
g_1(u,v) &= (\frac{1}{2}+v)V^K_1 + (\frac{1}{2}+u)(\frac{1}{2}-v)V^K_2+(\frac{1}{2}-u)(\frac{1}{2}-v)V^K_3\\
g_2(u,v) &= (\frac{1}{2}+v)V^K_2 + (\frac{1}{2}+u)(\frac{1}{2}-v)V^K_3+(\frac{1}{2}-u)(\frac{1}{2}-v)V^K_1\\
g_3(u,v) &= (\frac{1}{2}+v)V^K_3 + (\frac{1}{2}+u)(\frac{1}{2}-v)V^K_1+(\frac{1}{2}-u)(\frac{1}{2}-v)V^K_2.
\end{align*}
Let $\bar{U}_K$ denote the cell average of the numerical solution $U_h$ represented by the DG polynomial $p_K$ on triangle $K$, following \cite{ZhangShu12} it holds that 
\begin{align*}
\bar{U}_K = \frac{2}{3}\sum\limits_{i=1}^{3}\sum\limits_{\alpha=1}^{3}\sum\limits_{\beta=1}^{3}p_K\left(g_i\left(\hat{u}^\alpha,v^\beta\right)\right)\left(\frac{1}{2}-v^\beta\right)w^\alpha\hat{w}^\beta.
\end{align*}
Note that the Gauss-Lobatto points on the edge are always taken twice (see Figure \ref{fig:Quadrature}).
\begin{figure}[htbp]
\centering
\settikzlabel{fig:quadrulestrianglePos}
\settikzlabel{fig:quadrulestriangleWeights}
\externaltikz{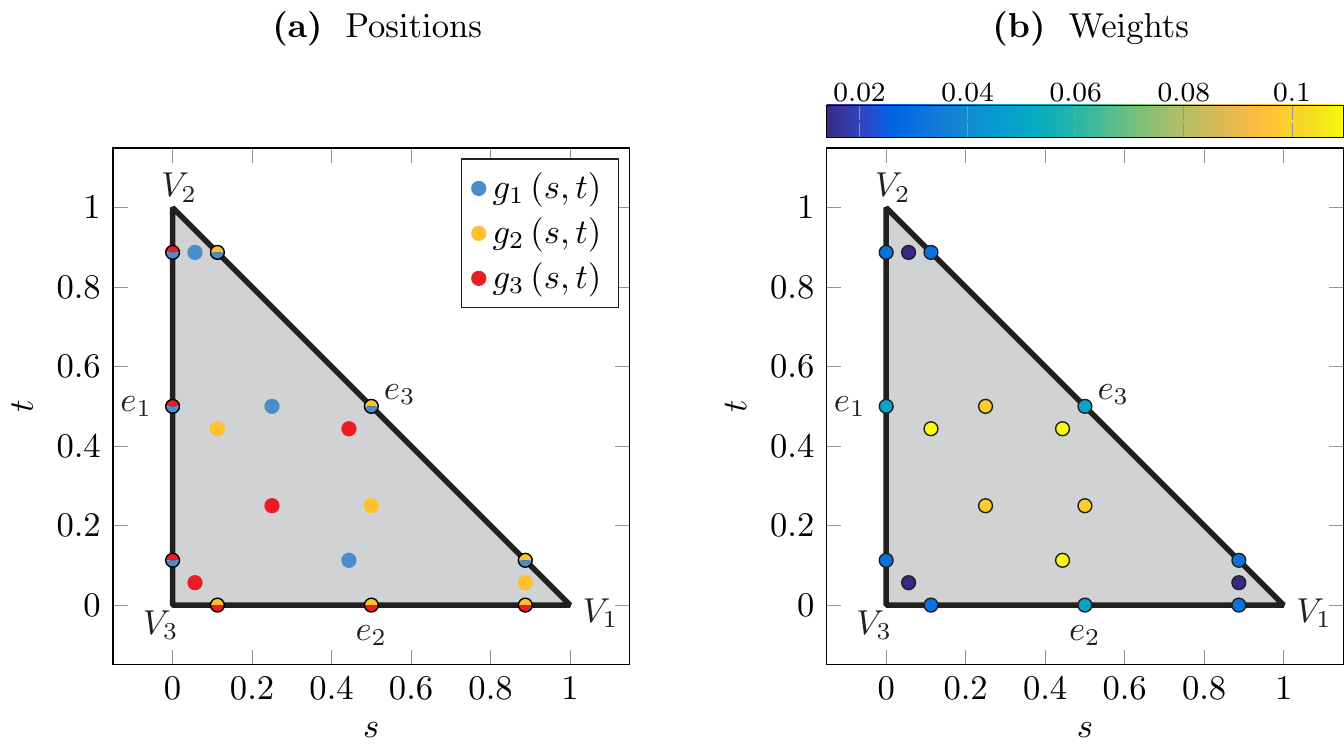}{\input{Images/quadrulestriangle}}  
\caption{Position of the quadrature nodes under the three projections $g_i$. The color corresponds to the weight $\tilde{w}^\gamma$ at the quadrature node.}
\label{fig:Quadrature}
\end{figure}

The authors of \cite{ZhangShu12} showed that the cell mean can be constructed as the following convex combination of inner and boundary points:
\begin{align}
\label{eq:Cellmean}
\bar{U}_{K} = \sum\limits_{i=1}^{3}\frac{2}{3}w^\beta\hat{w}^{\alpha=1} U_{i,\beta}^{int}+\sum\limits_{\gamma=1}^{L}\tilde{w}^\gamma U_\gamma^{inner}
\end{align}
where $U_{\gamma}^{inner}$ and $U_{i,\beta}^{int}$ represent the evaluation of the DG polynomial $p_K$ at the respective interior and boundary quadrature nodes 
$(x_{\gamma}^{inner},y_{\gamma}^{inner})$ and $(x_{i,\beta}^{int},y_{i,\beta}^{int})$ (compare Figure~\ref{fig:Quadrature}), and $L = 3(N-2)(k+1) = 9$ is the number of inner points. With this we obtain a quadrature rule with $18$ points which is accurate for polynomials of order $2k-1=3$.

The flux term in \eqref{eq:dweakform2} should be approximated by the $1D$ $(k+1)$ point Gauss quadrature with weights $w^\beta$:
\begin{align}
\label{eq:FluxGauss}
\sum\limits_{e\in \partial K}\int_e H_{e,K}(x,t) \varphi_i^K(x)~d\Gamma = \sum\limits_{i=1}^3\sum\limits_{\beta = 1}^{3} H_{e_i,K}(U_{i,\beta}^{int},U_{i,\beta}^{ext},t) \varphi_i^K(x_{i,\beta}^{int},y_{i,\beta}^{int})w^\beta l^i_K
\end{align}
where $l^i_K$ is the length of edge $e_i^K$. 

\subsubsection{Rectangles}
For the assembly of the discrete operator in the rectangular DG code we use a tensor quadrature rule of the 1D  $4$-point Gauss-Lobatto rule on $[-\frac12,\frac12]$:
$$
u^\alpha = \cfrac{1}{2\sqrt{5}}\begin{pmatrix}
-\sqrt{5}&-1&1&\sqrt{5}
\end{pmatrix},\quad
\hat{w}^\alpha = \frac{1}{12} \begin{pmatrix}
1 & 5 & 5 & 1
\end{pmatrix}.
$$
However for the proof of our main theorem we only need the three point rule \eqref{eq:GLrule} which will give a weaker CFL-condition.

\subsection{Slope Limiting}
In order to enforce stability \cite{Cockburn,hesthaven2007nodal} and to mitigate numerical oscillations caused by the Gibbs phenomenon we apply a slope limiter
to every stage of the Runge-Kutta time stepping scheme. We implement the slope limiter detailed 
by Cockburn and Shu in \cite{CockburnShuIV}
\begin{gather}
	\overline{m}(a_1,a_2,\cdots, a_n) = \begin{cases} a_1 & \mbox{ if } |a_1| < M (\Delta x)^2 \\
										m(a_1,a_2,\cdots,a_n) & \mbox{ otherwise,}
\end{cases}
\label{eq:dg_slope_limiter}
\end{gather}
where $M$ is an estimate of the second derivative of the solution, $\Delta x$ is the characteristic length of
each element and $\overline{m}$ is the TVB modified
minmod function. The limiter can be applied component-wise to the primitive variables
or in the characteristic variables. For limiting in the characteristic variables we first construct a matrix
$\mathcal{R}$ that diagonalizes the directional Jacobian in the normal direction $\bfn = (n_x,n_y)$ evaluated at the mean in each element
\begin{gather}
\label{eq:trafo}
	\mathcal{R}^{-1} (n_x J_x(\overline{U}_k) + n_y J_y(\overline{U}_k))\mathcal{R} =  \Lambda.
\end{gather}
Applying the limiter in the characteristic variables ensures that the solution is total variation bounded in the means (TVBM).

\section{Realizability Limiter}
\label{sec:limiter}
We want to construct a scheme so that the numerical solution stays realizable with respect to the cell means. The proofs presented in this section follow the strategy used in {\cite{Zhang2010}, which has also been used in the construction of realizability-preserving limiters for the $M_1$ model in 1D \cite{Olbrant12}. 

First we need a technical lemma.
\begin{lem}
\label{lem:Lemma3}
Let $\nu\in \mathbb{R}^2$ be an arbitrary unit vector. For the $M_1$ model the combination of moments $(\psiz\pm \psio\cdot\nu,\psio\pm \psit\cdot \nu)$ is realizable.
\end{lem}
\begin{proof}
Let $\psi$ be a non-negative density that realizes $\psiz$ and $\psio$. Then $(1\pm\nu\cdot\Omega)\psi(\Omega)$ is a non-negative density that realizes the combined moments.
\end{proof}

Let us now consider the higher-order scheme \eqref{eq:dweakform2}. Due to the convexity of the realizable set it suffices to investigate the forward-Euler discretization in time since the used SSP integrator is just a convex combination of such Euler steps.
\begin{thm}[Main Result]
\label{thm:MainResult}
One forward-Euler step of the scheme \eqref{eq:dweakform2} with the DG polynomial $p_K$ of degree $k$ yields realizable cell averages
if $p_K(x_s,y_s)$ is realizable for all $(x_s,y_s)\in S_k^K$ and if the CFL condition
\begin{gather}
\frac{2}{3}\hat{w}_1\left(1-\Delta t\left(\sigma_a + \sigma_s\right)\right) - \frac{\Delta t}{2\abs{K}}l_K^i\geq 0~~\forall i=1,\ldots,3
\label{eqn:cfl}
\end{gather}
holds. Here $\hat{w}_1$ is the quadrature weight of the $N$-point Gauss-Lobatto rule on $[-\frac{1}{2},\frac{1}{2}]$ for the first quadrature point.
\end{thm}
\begin{proof}
    For simplicity assume that $\sigma_s=\sigma_a = 0$ and $q = 0$. Furthermore, to ease notation, we drop the symbol $n$ for all quantities at the current time step. After application of the forward-Euler scheme the cell averages in element $K$ satisfy 
\begin{align*}
\overline{U}_K^{n+1} &= \overline{U}_K - \frac{\Delta t}{\abs{K}}\sum\limits_{i=1}^{3}\int_{e_i} H_{e_i,K}\left(U_{K}^{int},U_{K(i)}^{ext}\right)~dx\\
&\stackrel{\eqref{eq:LFFlux}}{=}  \begin{pmatrix}\psizbar_K\\\psiobar_K\end{pmatrix} - \frac{\Delta t}{2\abs{K}}\sum\limits_{i=1}^{3}\int_{e_i} \begin{pmatrix}\psiz_K\\\psio_K\end{pmatrix}~dx + \frac{\Delta t}{2\abs{K}}\sum\limits_{i=1}^{3}\int_{e_i} \begin{pmatrix}\psiz_{K(i)}-\psio_{K(i)}\cdot \bfn_{e_i,K}\\\psio_{K(i)} - \psit_{K(i)}\cdot \bfn_{e_i,K}\end{pmatrix}~dx,
\end{align*}
where $U_K = p_K$ and $U_{K(i)} = p_{K(i)}$ denote the polynomial representations of $U_h$ on element $K$ and on its neighbor $K(i)$ sharing edge $e_i$ with $K$, respectively.

Evaluating the second integral using the Gaussian quadrature as in \eqref{eq:FluxGauss} it is easy to see by using Lemma \ref{lem:Lemma3} that on each edge $e_i$ the quantity
$$
\int_{e_i} \begin{pmatrix}\psiz_{K(i)}-\psio_{K(i)}\cdot \bfn_{e_i,K}\\\psio_{K(i)} - \psit_{K(i)}\cdot \bfn_{e_i,K}\end{pmatrix}~dx 
=  
l_K^i\sum_{\beta=1}^{k+1} w^\beta \begin{pmatrix}\psiz_{K(i),\beta}-\psio_{K(i),\beta}\cdot \bfn_{e_i,K}\\\psio_{K(i),\beta} - \psit_{K(i),\beta}\cdot \bfn_{e_i,K}\end{pmatrix}
$$
is realizable since $\sum w^\beta = 1$. Here, $\psiz_{K(i),\beta}$ denotes the evaluation of the polynomial representation of $\psiz$ on the neighboring element $K(i)$ at the quadrature node associated with $w^\beta$ on $e_i$ and similarly for all other moments.

Now it suffices to show that the first terms in this equation are realizable, too. We rewrite the cell average into a combination of interior points and edge points:
\begin{align*}
\overline{U}_K &\stackrel{\eqref{eq:Cellmean}}{=} \sum\limits_{i=1}^{3}\sum\limits_{\beta=1}^{k+1}\frac{2}{3}w^\beta\hat{w}^1 U_{i,\beta}^{int} + \sum\limits_{\gamma=1}^{L}\tilde{w}^\gamma U_\gamma^{inner},\\
\int_{e_i} U_K~dx  &= \sum\limits_{\beta=1}^{k+1}w^\beta l_K^i U_{i,\beta}^{int}.
\end{align*}
Then
\begin{align*}
\begin{pmatrix}\psizbar_K\\\psiobar_K\end{pmatrix} - \frac{\Delta t}{2\abs{K}}\sum\limits_{i=1}^{3}\int_{e_i} \begin{pmatrix}\psiz_K\\\psio_K\end{pmatrix}~dx = \sum\limits_{\gamma=1}^{L}\tilde{w}^\gamma U_\gamma^{inner} + \sum\limits_{i=1}^{3}\sum\limits_{\beta=1}^{k+1}w^\beta U_{i,\beta}^{int} \left(\frac{2}{3}\hat{w}^1 - \frac{\Delta t}{2\abs{K}}l_K^i\right).
\end{align*}
Under the CFL condition, we immediately get that 
$$
\sum\limits_{\gamma=1}^{L}\tilde{w}^\gamma U_\gamma^{inner} + \sum\limits_{i=1}^{3}\sum\limits_{\beta=1}^{k+1}w^\beta U_{i,\beta}^{int} \left(\frac{2}{3}\hat{w}^1 - \frac{\Delta t}{2\abs{K}}l_K^i\right)
$$
is realizable since the quadrature weights sum up to $1$.\\ 
For absorption and scattering we assume for simplicity that $\sigma_s$ and $\sigma_a$ are constant in $K$. Then we can write the cell averages of the source term as 
\begin{align*}
\overline{S(U_K)} &=  \sum\limits_{i=1}^{3}\sum\limits_{\beta=1}^{k+1}\frac{2}{3}w^\beta\hat{w}^1 \left(\sigma_a U_{i,\beta}^{int} + \sigma_s \begin{pmatrix}
0\\\psio_{i,\beta,int}
\end{pmatrix} \right)\\&+ \sum\limits_{\gamma=1}^{L}\tilde{w}^\gamma \left(\sigma_a U_\gamma^{inner} + \sigma_s \begin{pmatrix}
0\\\psio_{\gamma,{inner}}
\end{pmatrix} \right).
\end{align*}
Thus the updated cell averages can be written as the moments of a non-negative distribution, i.e. $U_K^{n+1} = \int \begin{pmatrix}
1\\\Omega
\end{pmatrix}\Xi_K(\Omega)~d\Omega$ with
\begin{align*}
\Xi_K &=\sum\limits_{\gamma=1}^{L}\tilde{w}^\gamma \left[\left(1-\Delta t\left(\sigma_a + \sigma_s\right)\right)\psi_\gamma^{inner} + \cfrac{\sigma_s}{4\pi}\int \psi_\gamma^{inner}~d\Omega \right] \\
&+ \sum\limits_{i=1}^{3}\sum\limits_{\beta=1}^{k+1}w^\beta \left[\psi_{i,\beta}^{int} \left(\frac{2}{3}\hat{w}^1\left(1-\Delta t\left(\sigma_a + \sigma_s\right)\right) - \frac{\Delta t}{2\abs{K}}l_K^i\right) +\frac{\Delta t}{2\abs{K}}l_K^i {\psi}_{K(i),\beta}+ \cfrac{\sigma_s}{4\pi}\int \psi_{i,\beta}^{int} ~d\Omega\right]
\end{align*}
The quantity $\psi_\gamma^{inner}$ is any nonnegative distribution function realizing $U_{\gamma}^{inner}$, similarly for $\psi_{i,\beta}^{int}$ and ${\psi}_{K(i),\beta}$. It is easy to see that under the given assumptions $\Xi_K^n\geq 0$, which implies that $U_K^{n+1}$ is by definition realizable.
\end{proof}

A similar result can be obtained on the rectangular grid by going through the same lines of the previous proof, replacing the quadrature rules on the triangle with those on the rectangles. More details can be found in \cite{Schneider2016}.

\begin{thm}[Main Result for rectangular grid]
\label{thm:MainResultRect}
One forward-Euler step of the scheme \eqref{eq:dweakform2} with the DG polynomial $p_K$ of degree $k$ yields realizable cell averages
if $p_K(x_s,y_s)$ is realizable for all $(x_s,y_s)\in S_k^K$ and if the CFL condition
$$
\hat{w}_1\left(1-\Delta t\left(\sigma_a + \sigma_s\right)\right) - \frac{\Delta t}{\Delta x}- \frac{\Delta t}{\Delta y}\geq 0
$$
holds. Here $\hat{w}_1$ is the quadrature weight of the $N$-point Gauss-Lobatto rule on $[-\frac{1}{2},\frac{1}{2}]$ for the first quadrature point.
\end{thm}

All that remains is to ensure that the assumptions of the two theorems are satisfied. This can be achieved by a simple scaling limiter. A similar limiter has been derived in other contexts, e.g. shallow water \cite{Xing2013}, Euler equations \cite{Zhang2010} and gas dynamics \cite{Zhang2012a}. The basic idea is to dampen the higher-order parts of the DG polynomial $p_K(x,y)$ until at all quadrature nodes $(x_s,y_s)\in S_k^K$ the evaluation $p_K(x_s,y_s)$ is realizable. It has been reported in \cite{Zhang2012a} that this limiter, due to its simplicity, can destroy the formal accuracy of the scheme in certain non-generic situations. This has been investigated in a one-dimensional setting for minimum-entropy models in \cite{Schneider2015a}. There, it has been shown that close to the realizability boundary the convergence order can drop.  However, we still expect order preservation in most practical cases.

Writing the limited polynomial as 
\begin{align*}
p_K^\theta(x,y) := \theta \overline{p_K} + (1-\theta)p_K(x,y) \stackrel{\eqref{eqn:solution_form}}{=} \theta \alpha_1^U\varphi_1 +  (1-\theta)\sum\limits_{i=1}^{N_k} \alpha_i^U\varphi_i = \alpha_1^U\varphi_1 + (1-\theta)\sum\limits_{i=2}^{N_k} \alpha_i^U\varphi_i
\end{align*}
it becomes clear that the cell average of $p_K$ is preserved and only higher-order coefficients are damped by the factor $(1-\theta)$ where $\theta$ is chosen as the minimal value in the set
\begin{align}
\left\{\tilde{\theta}\in[0,1]~|~p_K^{\tilde\theta}(x_s,y_s)~\text{ is realizable for all } (x_s,y_s)\in S_k^K \right\}.
\label{eq:RLproblem}
\end{align}
Such a $\theta$ always exist under the assumption that $\overline{p_K}$ is realizable as a consequence of the convexity of the realizable set. To calculate $\theta_s$ it suffices to find the intersections of the ray segments
\begin{align}
\label{eq:RaySegment}
\left\{\theta \bar U_K +(1-\theta) p_K(x_s,y_s)~|~\theta\in[0,1]\right\},~(x_s,y_s)\in S_k^K
\end{align} 
with the boundary of the realizable set
\begin{align}
\label{eq:BoundaryReal}
\left\{\left(\psi^{(0)},\psi^{(1)}\right)\in\mathbb{R}^3~|~ \psi^{(0)} = \left\lvert \psi^{(1)}\right\rvert \right\} = \left\{\left(\psi^{(0)},\psi^{(1)}\right)\in\mathbb{R}^3~|~ \psi^{(0)}\geq 0, \left(\psi^{(0)}\right)^2 = \left\lvert\psi^{(1)}\right\rvert^2 \right\}.
\end{align}
Using the condition $\left(\psi^{(0)}\right)^2 = \left\lvert\psi^{(1)}\right\rvert^2$ from \eqref{eq:BoundaryReal} and the parameterization of the ray segment \eqref{eq:RaySegment} we get
\begin{align*}
\left(\theta\bar\psi^{(0)}+(1-\theta)\psi^{(0)}\right)^2 \stackrel{!}{=} \left(\theta\bar\psi^{(1)}_y+(1-\theta)\psi^{(1)}\right)^2+\left(\theta\bar\psi^{(1)}_y+(1-\theta)\psi^{(1)}\right)^2.
\end{align*} 
Rearranging this expression, collecting the coefficients in powers of $\theta$ yields the polynomial equation
\begin{align*}
0\stackrel{!}{=} r(\theta) &= a\theta^2+b\theta+c,\\
a &= -\left(\bar{\psi}^{(0)}-{\psi}^{(0)}\right)^2 + \left(\bar{\psi}^{(1)}_x -{\psi}^{(1)}_x\right)^2 + \left(\bar{\psi}^{(1)}_y -{\psi}^{(1)}_y\right)^2\\
b &= 2\cdot\left(\left({\psi}^{(0)}\right)^2 - \bar{\psi}^{(0)}\cdot{\psi}^{(0)} - \left({\psi}^{(1)}_x\right)^2 + \bar{\psi}^{(1)}_x\cdot{\psi}^{(1)}_x - \left({\psi}^{(1)}_y\right)^2 + \bar{\psi}^{(1)}_y\cdot{\psi}^{(1)}_y\right)\\
c &= - \left({\psi}^{(0)}\right)^2 + \left({\psi}^{(1)}_x\right)^2 + \left({\psi}^{(1)}_y\right)^2.
\end{align*}
Since $r(\theta)$ is a quadratic polynomial, it possibly has two zeros in the set $[0,1]$. Due to convexity, the largest zero in this set is the correct choice for $\theta_s$, since it corresponds to the intersection with the $\psi^{(0)}\geq 0$ part of the double cone $|\psi^{(1)}|^2\leq (\psi^{(0)})^2$, while the smaller zero (which is by definition farther away from $\bar U_K$) corresponds to the intersection with the  $\psi^{(0)}\leq 0$ part of the double cone. Finally, set $\theta := \max_s \theta_s$.

Note that in each intermediate step of the Runge-Kutta scheme the realizability limiter is always applied after the slope limiting. Applying the slope limiter can destroy the realizability again \cite{Schneider2016}.

\section{Numerical Results}
\label{sec:numerical_results}
In this section, we compare the performance of the RKDG schemes on unstructured triangular meshes
and uniform rectangular meshes for various combinations of limiters. Specifically we consider the limiter combinations given in Table~\ref{tab:LimiterCombinations}.
The time step is taken as indicated by~\eqref{eqn:cfl}.
\begin{table}[h]
\centering
\begin{tabular}{l l}
Abbr. & Explanation\\
\midrule
SL$M$$\square/\triangle$ & Slope limiter in primitive variables with the constant $M$ specified\\
CL$M$$\square/\triangle$ & Slope limiter in characteristic variables with the constant $M$ specified\\
SRL$M$$\square/\triangle$ & Slope limiter in primitive variables + realizability limiter\\
CRL$M$$\square/\triangle$ & Slope limiter in characteristic variables + realizability limiter\\
\end{tabular}
\caption{Abbreviation of limiter combinations. Setting $M=\infty$ is equivalent to disabling the slope limiter. $\square$ and $\triangle$ correspond to the rectangular and triangular meshes, respectively.}
\label{tab:LimiterCombinations}
\end{table}
To ensure that our research is reproducible, the codes used in this publication can be found on \textit{GitHub}, see \cite{Chidyagwai2016}.

In the numerical results that follow we denote the mesh size on both unstructured triangular meshes
and uniform rectangular meshes by $h$. All plots show either the zeroth moment $\psi^{(0)}$ or the norm of the first normalized moment $\vert\phi^{(1)}\vert = \vert\frac{\psi^{(1)}}{\psi^{(0)}}\vert$. The loss of realizability is indicated in the latter by the use of black ($\vert\phi^{(1)}\vert>1$) and white ($\psi^{(0)}<0$) colors. 

Before we describe the test cases and the results, we detail some of the steps in the algorithm, and the computation of the errors.

{\bfseries Fixing a DG solution which is not realizable.}
In test cases where the realizability limiter is not used (SL$M$/CL$M$) the solution may violate the realizability condition and thus the flux function cannot 
be evaluated. We therefore compute a modified numerical solution which we use only to compute the flux, i.e.\ the actual solution in each cell remains the same. 
We modify the values of the numerical solution locally at each Gauss point
as follows: let $\epsilon$ be a specified numerical tolerance, if $\psi^{(0)} < \epsilon$ then 
we set $\psi^{(0)} = \epsilon$. In addition, if the normalized first moment 
$$
f :=  |\frac{\psi^{(1)}}{\psi^{(0)}}| > (1 - \epsilon),
$$ 
we modify the values of first moment as follows: $\displaystyle \psi^{(1)}_x = \dfrac{\psi^{1}_x}{f}(1-\epsilon)$ and 
$\displaystyle \psi^{(1)}_y = \dfrac{\psi^{1}_y}{f}(1-\epsilon)$. We choose a tolerance of $\epsilon = 10^{-12}$.

{\bfseries Transformation to characteristic variables.}

In the case of the triangular meshes the matrices $\mathcal{R}$ and $\mathcal{R}^{-1}$ from \eqref{eq:trafo} are computed using the 
GSL linear algebra package \cite{Gough_GSL} while they can be precomputed analytically in case of the rectangular grid. 
The transformation to and from the characteristic variables is achieved by left multiplying the variables by $\mathcal{R}^{-1}$ and $\mathcal{R}$, respectively. 

In case that we need the transformation to characteristic variables and do not use the realizability limiter (i.e. CL$M$), we use the same realizability fix as in the flux function, but for the computation of the transformation matrices only. This ensures that transformation matrices are invertible due to the strict hyperbolicity of the moment system in the interior of the realizable set. The transformation matrix has a condition number of roughly $2/\epsilon$, where $\epsilon$ is the distance of the (cell mean of the) normalized first moment to the realizability boundary. Note that this implies that close to the realizability boundary, significant round-off errors may occur.
However, we have performed numerical experiments with the linesource test case with $\epsilon=10^{-5}, 10^{-12}, 10^{-14}$, and observed no significant effect of the cutoff parameter.

{\bfseries Comparison to a reference solution.}
We note that for a general $f>0$ we have that $$\lVert\log_{10}(f)\rVert_{H^1(K)} = \lVert\log_{10}(f)\rVert_{L^2(K)}+\lVert\frac{1}{\log_{10}f}\nabla_xf\rVert_{L^2(K)}.$$ 
We can therefore use the evaluation of the DG polynomial and its first-order derivatives in every element $K$ to calculate its logarithmic Sobolev representation locally.
 
The reference solution is computed using a first-order Lax-Friedrichs finite-volume scheme on an equidistant rectangular grid with $8192\times 8192$ grid points. The gradients for the reference solution are obtained using centered finite difference formulas. Finally, the argument in the norm is evaluated on every point of the reference grid and integrated using a quadrature with equal weights at exactly those reference points (rectangle rule).

{\bfseries Choosing the limiter parameter $M$.}
The solutions of the benchmark test cases that follow are plotted on a logarithmic scale. We thus choose the value of $M$ in \eqref{eq:dg_slope_limiter}, using the logarithmic 
Sobolev norm error, as
\begin{align}
\label{eq:logSob}
M = \operatorname{arg min}\limits_{M\geq 0} \lVert \log_{10}\left(\psi^{(0)}_{M}\right)-\log_{10}\left(\psi^{(0)}_{ref}\right)\rVert_{H^1(X)}.
\end{align}
Since every evaluation of \eqref{eq:logSob} requires to solve the full system of equations, finding the true minimum using e.g. a gradient method is not feasible. We therefore restricted ourself to the discrete set of (almost) logarithmically spaced values $M\in\mathcal{M}:=\{0.1,0.2,0.5,1,2,10,22,46,100,150\}$.

\subsection{Failure of the realizability limiter}
We want to extend the investigation in \cite{Schneider2015a} to our two-dimensional setup. 
Define the two moment vectors 
\begin{align*}
U_0 &= (1-\xi)\begin{pmatrix}
1\\1\\0
\end{pmatrix} + \xi\begin{pmatrix}
1\\0\\0
\end{pmatrix},&
U_1 = 10^{-6}\cdot\left((1-\xi)\begin{pmatrix}
1\\0\\1
\end{pmatrix} + \xi\begin{pmatrix}
1\\0\\0
\end{pmatrix}\right).
\end{align*}
Both of them satisfy $\lvert \frac{\psio}{\psiz}\rvert = 1-\xi$, i.e. the distance to the realizability boundary can be controlled by the parameter $\xi$. Due to the convexity of the realizable set every convex combination of $U_0$ and $U_1$ will be realizable as well. We can thus define the vector of functions
\begin{align}
\label{eq:u-limit-test}
U(x,y) &= (1-\lambda(x,y))U_0 + \lambda(x,y)U_1,\\
\lambda(x,y) &= \cfrac{\cos\left(2\left(x+y\right)\pi\right)+1}{2}\in[0,1],\nonumber
\end{align}
which is realizable for all $(x,y)\in[0,1]^2$.

Now we discretize the domain, for simplicity, using the structured rectangular mesh with equidistant step sizes. We project $U(x,y)$ onto the DG basis of degree $k$, $k\in\{0,1,2\}$ and apply the realizability-preserving limiter from Section~\ref{sec:limiter}.
Then we refine the grid and calculate $L_1$- and $L_\infty$-errors for the zeroth moment $\psiz$ as 
\begin{align*}
E^1_h = \int_{0}^{1}\int_{0}^{1} \abs{\psiz_a(x,y)-\psiz_h(x,y)}~dxdy, \quad E^\infty_h = \max\limits_{(x,y)\in[0,1]^2}\abs{\psiz_a(x,y)-\psiz_h(x,y)},
\end{align*}
where $\psiz_a$ denotes the exact zeroth moment of $U$ and $\psiz_h$ its limited DG polynomial.
In practice, the integral in $E^1_h$ is approximated by a high-order tensor Gaussian quadrature on every cell and the $\max$ in $E^\infty_h$ is taken over the same quadrature nodes.

The observed convergence order $\nu$ is defined by
\begin{equation}
 \frac{E^p_{h1}}{E^p_{h2}} = \left( \frac{h_1}{h_2} \right)^\nu
\label{eq:conv-order}
\end{equation}
where for $i \in \{1, 2\}$, $E^p_{hi}$ is the error $E^p_h$ for the
numerical solution using cell size $h_i$, for $p \in \{1, \infty\}$.

We found that taking $\xi \in [0, 10^{-4}]$ places the moment curve
$U(x,y)$ close enough to the boundary of realizability that the realizability
limiter was active for every number of cells we considered.
In Table~\ref{tab:limiter-test-4} we show convergence rates for 
$\xi=10^{-4}$. These results show the expected convergence order $k+1$ in all cases.
In this table we include the column $\theta_{\max}$, which gives the
maximum value of $\theta$ from the realizability limiter over all spatial
cells.
That $\theta_{\max}$ is nonzero in each row indicates that the
realizability limiter was active for every reconstruction.

\begin{table}
\centering
\begin{tabular}{r r@{.}l c r@{.}l c r@{.}l r@{.}l c r@{.}l c r@{.}l}

 & \multicolumn{8}{c}{$k=1$} & \multicolumn{8}{c}{$k=2$}\\
\cmidrule(r){2-9} \cmidrule(r){10-17}
$1/h$ & \multicolumn{2}{c}{$E^1_h$} & $\nu$
 & \multicolumn{2}{c}{$E^\infty_h$} & $\nu$
 & \multicolumn{2}{c}{$\theta_{\max}$}
 & \multicolumn{2}{c}{$E^1_h$} & $\nu$
 & \multicolumn{2}{c}{$E^\infty_h$} & $\nu$
 & \multicolumn{2}{c}{$\theta_{\max}$} \\ \midrule
 
 5 & 5 & 535e-02 & --- & 2 & 790e-01 & --- & 3 & 210e-01& 1 & 787e-02 & --- & 6 & 567e-02 & --- & 1 & 783e-01\\
10 & 1 & 252e-02 & 2.1 & 1 & 269e-01 & 1.1 & 3 & 960e-01& 1 & 483e-03 & 3.6 & 1 & 153e-02 & 2.5 & 5 & 305e-02\\
20 & 2 & 774e-03 & 2.2 & 3 & 831e-02 & 1.7 & 4 & 117e-01& 1 & 382e-04 & 3.4 & 1 & 495e-03 & 2.9 & 1 & 391e-02\\
40 & 6 & 500e-04 & 2.1 & 1 & 001e-02 & 1.9 & 4 & 154e-01& 1 & 551e-05 & 3.2 & 1 & 878e-04 & 3.0 & 3 & 519e-03\\
80 & 1 & 568e-04 & 2.1 & 2 & 531e-03 & 2.0 & 4 & 164e-01& 1 & 881e-06 & 3.0 & 2 & 350e-05 & 3.0 & 8 & 824e-04\\
160 & 3 & 847e-05 & 2.0 & 6 & 345e-04 & 2.0 & 4 & 166e-01& 2 & 332e-07 & 3.0 & 2 & 938e-06 & 3.0 & 2 & 205e-04\\
320 & 9 & 526e-06 & 2.0 & 1 & 587e-04 & 2.0 & 4 & 166e-01& 2 & 909e-08 & 3.0 & 3 & 673e-07 & 3.0 & 5 & 431e-05\\

\end{tabular}
\caption{$L^1$- and $L^\infty$-errors and observed convergence order $\nu$
for the first component of the realizability-limited, piece-wise linear and
quadratic reconstructions of $U(x,y)$ from \eqref{eq:u-limit-test} with
$\xi = 10^{-4}$.}
\label{tab:limiter-test-4}
\end{table}

Pushing $U$ closer to the realizability boundary ($\xi\leq 10^{-5}$) degrades the convergence order of the piecewise-linear reconstruction, verifying the mentioned convergence problems. This is demonstrated in Table~\ref{tab:limiter-test-0} for the extreme case $\xi = 0$, i.e. moments are placed on the realizability boundary. Surprisingly, the third-order approximation has the full convergence order for all tested $\xi\in[0,10^{-3}]$.
Note that this is different to the results in \cite{Schneider2015a} where the third-order approximation degrades to second order. The reason for this is still unclear but might be due to the high non-linearity of the $M_3$ model used in \cite{Schneider2015a}. We checked several variants of the test (by choosing different vectors on the unit sphere instead of $(1,0,0)^T$ and modifying the weight $10^{-6}$ in front of $U_1$) and all showed the same behavior. 

\begin{table}
\centering
\begin{tabular}{r r@{.}l c r@{.}l c r@{.}l r@{.}l c r@{.}l c r@{.}l}

 & \multicolumn{8}{c}{$k=1$} & \multicolumn{8}{c}{$k=2$}\\
\cmidrule(r){2-9} \cmidrule(r){10-17}
$1/h$ & \multicolumn{2}{c}{$E^1_h$} & $\nu$
 & \multicolumn{2}{c}{$E^\infty_h$} & $\nu$
 & \multicolumn{2}{c}{$\theta_{\max}$}
 & \multicolumn{2}{c}{$E^1_h$} & $\nu$
 & \multicolumn{2}{c}{$E^\infty_h$} & $\nu$
 & \multicolumn{2}{c}{$\theta_{\max}$} \\ \midrule
 
 5 & 6 & 716e-02 & --- & 3 & 594e-01 & --- & 5 & 603e-01& 1 & 787e-02 & --- & 6 & 567e-02 & --- & 1 & 783e-01\\
10 & 1 & 532e-02 & 2.1 & 1 & 269e-01 & 1.5 & 3 & 960e-01& 2 & 033e-03 & 3.1 & 1 & 153e-02 & 2.5 & 5 & 305e-02\\
20 & 3 & 143e-03 & 2.3 & 3 & 831e-02 & 1.7 & 4 & 117e-01& 1 & 577e-04 & 3.7 & 1 & 495e-03 & 2.9 & 1 & 391e-02\\
40 & 6 & 966e-04 & 2.2 & 1 & 001e-02 & 1.9 & 4 & 154e-01& 1 & 614e-05 & 3.3 & 1 & 878e-04 & 3.0 & 3 & 519e-03\\
80 & 1 & 626e-04 & 2.1 & 2 & 531e-03 & 2.0 & 4 & 164e-01& 1 & 901e-06 & 3.1 & 2 & 350e-05 & 3.0 & 8 & 829e-04\\
160 & 3 & 920e-05 & 2.1 & 6 & 345e-04 & 2.0 & 4 & 166e-01& 2 & 338e-07 & 3.0 & 2 & 938e-06 & 3.0 & 2 & 217e-04\\
320 & 9 & 625e-06 & 2.0 & 3 & 008e-04 & 1.1 & 4 & 167e-01& 2 & 912e-08 & 3.0 & 3 & 673e-07 & 3.0 & 8 & 294e-05\\
\end{tabular}
\caption{$L^1$- and $L^\infty$-errors and observed convergence order $\nu$
for the first component of the realizability-limited, piece-wise linear and
quadratic reconstructions of $U(x,y)$ from \eqref{eq:u-limit-test} with
$\xi = 0$.}
\label{tab:limiter-test-0}
\end{table}

Full convergence tests for this scheme in one and two dimensions can be found in \cite{Schneider2015a,Schneider2016}.

%%%%%%%%%%%%%%%%%%%%%%%%
%Linesource
\label{sec:results}
\subsection{Line Source}
%\comment{M = 22}
Our first benchmark is the line-source test as proposed in \cite{Brunner2005b}. It is a Green's function problem, where a pulse of particles is emitted from a line in an infinite medium. This corresponds to an initial condition of the form $\psi^{(0)}(x,y) = c\cdot\delta(x,y)$, whose lack of regularity makes this problem especially hard to tackle for the method of moments, including the $M_1$ model, see e.g. \cite{Garrett2014}. The parameters of this test are as follows:
\begin{itemize}
\item Domain: $[-0.5,0.5]^2$ with $h =0.004$
\item Time $T=0.45$
\item Parameters: $\sigma_a=\sigma_s=0$, $q^{(0)}=0$, $q^{(1)}=0$
\item Initial condition: smoothed Dirac $\psi^{(0)}(x,y) = \max(\exp(-10\frac{x^2+y^2}{\sigma^2}),10^{-4})$ with $\sigma=0.02$, $\psi^{(1)}(x,y)=0$
\item Boundary conditions are Dirichlet conditions consistent with the initial conditions. Because the signal does not reach the boundary, these conditions do not influence the result.
\end{itemize}

Due to the rotational invariance of the initial condition and the flux of the $M_1$ model the exact solution for the line-source problem has rotational symmetry.

\begin{figure}[h]
\def\localpath{./Images/Linesource/}
\centering
\externaltikz{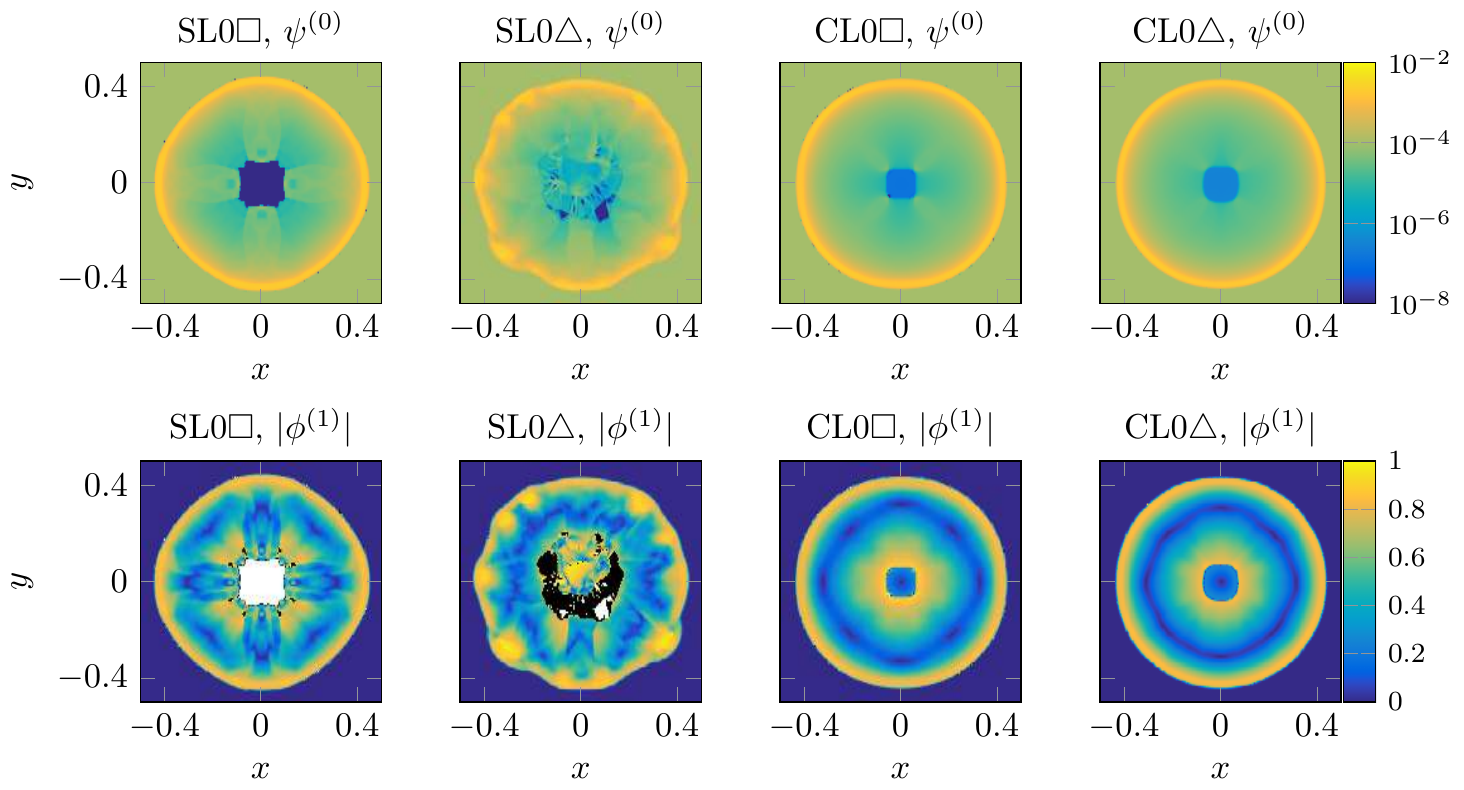}{\input{Images/Linesource4}}
\caption{Line source: comparison of primitive and characteristic slope limiter.}
\label{fig:linesource4}
\end{figure}

This symmetry is easily destroyed by the application of the slope limiter in primitive variables (SL$0$$\square/\triangle$, Figure~\ref{fig:linesource4}). 
It is visible for both mesh types that the solution is strongly asymmetric in density as well as normalized velocity. Furthermore, realizability is lost in parts of the domain (black and white spots).

Applying the characteristic limiter (CL$0$$\square/\triangle$) yields much better results: the solution is much more symmetric and as shown in 
Tables~\ref{tab:linesource_rect_table} and~\ref{tab:linesource_tri_table}, it is even realizable in the mean at the final time step. 
However, the realizability in the mean does not hold at every time step and in addition there are a Gauss points on both meshes at which the solution is unrealizable. 
We note that the fact that primitive-variable limiting (SL$0$$\square/\triangle$) can result in non-physical oscillations is a well-known phenomenon 
and has been reported in case of the $M_1$ model in one dimension in \cite{Olbrant12}. 

\begin{figure}[h]
\def\localpath{./Images/Linesource/}
\centering
\externaltikz{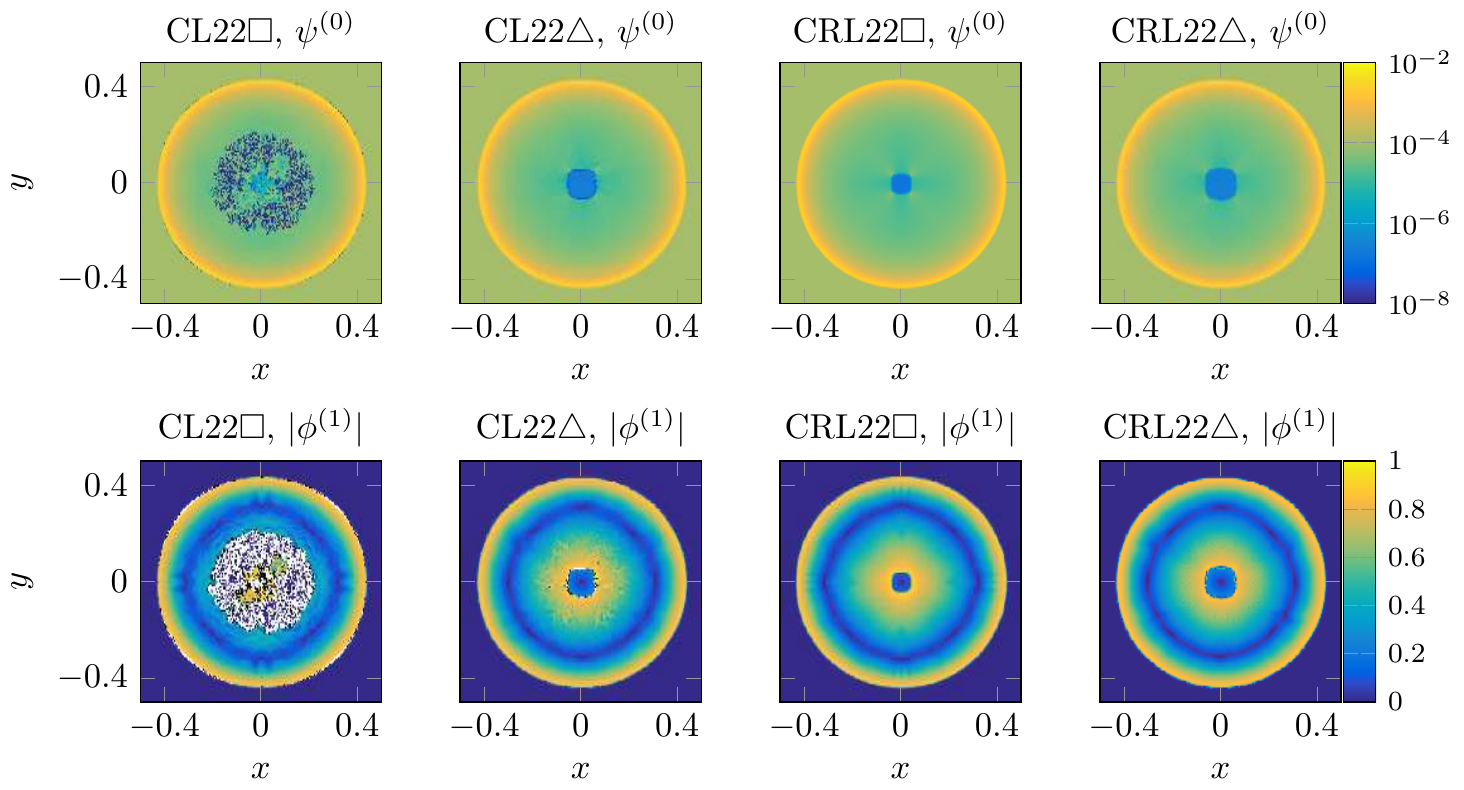}{\input{Images/Linesource4b}}
\caption{Line source: comparison of the characteristic limiter with and without realizability limiter.}
\label{fig:linesource4b}
\end{figure}

Unfortunately, choosing $M=0$ in the minmod limiter results in a flattening of the solution at smooth extrema and first-order convergence in the $L_\infty$-norm \cite{CockburnShuPk}. 
Setting $M=22$ results in sharper solutions in the sense that the width of the wave front is decreasing (Figure \ref{fig:linesource4b}). However, this comes at the price of a realizability loss resulting in a complete destruction of the solution (see e.g. CL$22\square$ in Figure~\ref{fig:linesource4b}). Similar results occur with the primitive-variable limiter (not shown).
This can be fixed by further application of the realizability limiter (CRL$22\square/\triangle$). The now-realizable solutions show the same sharpness as with CL$22\square/\triangle$. 
Furthermore, good rotational symmetry is still visible.

\begin{figure}[!h]
\def\localpath{./Images/Linesource/}
\centering
\externaltikz{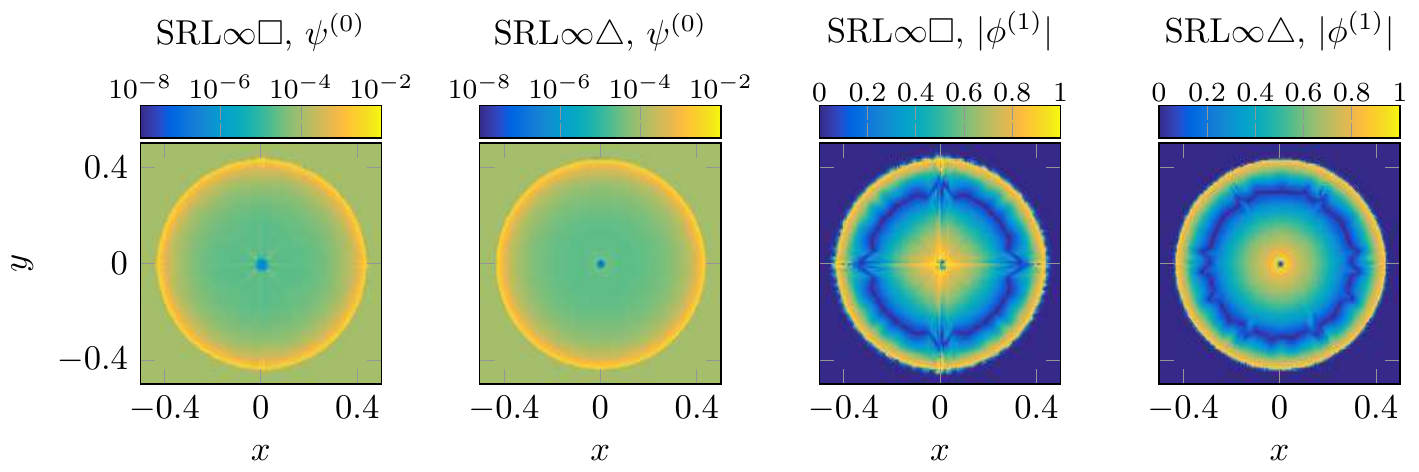}{\input{Images/Linesource2}}
\caption{Line source: realizability limiter only.}
\label{fig:linesource2}
\end{figure}

A natural question arising is if the realizability limiter on its own is enough to provide a good solution. It has been shown in \cite{Olbrant12} that even in one dimension spurious oscillations can occur. Similarly, as shown in Figure \ref{fig:linesource2}, symmetry is not preserved and slight oscillations are visible, especially in the normalized velocity. Although the solution is everywhere realizable, a slope limiter is strictly necessary to dampen the Gibbs phenomenon. Similar observations have been made in case of the Euler equations in \cite{Zhang2010,Wang2012}.

\begin{figure}[h]
\def\localpath{./Images/Linesource/}
\centering
\externaltikz{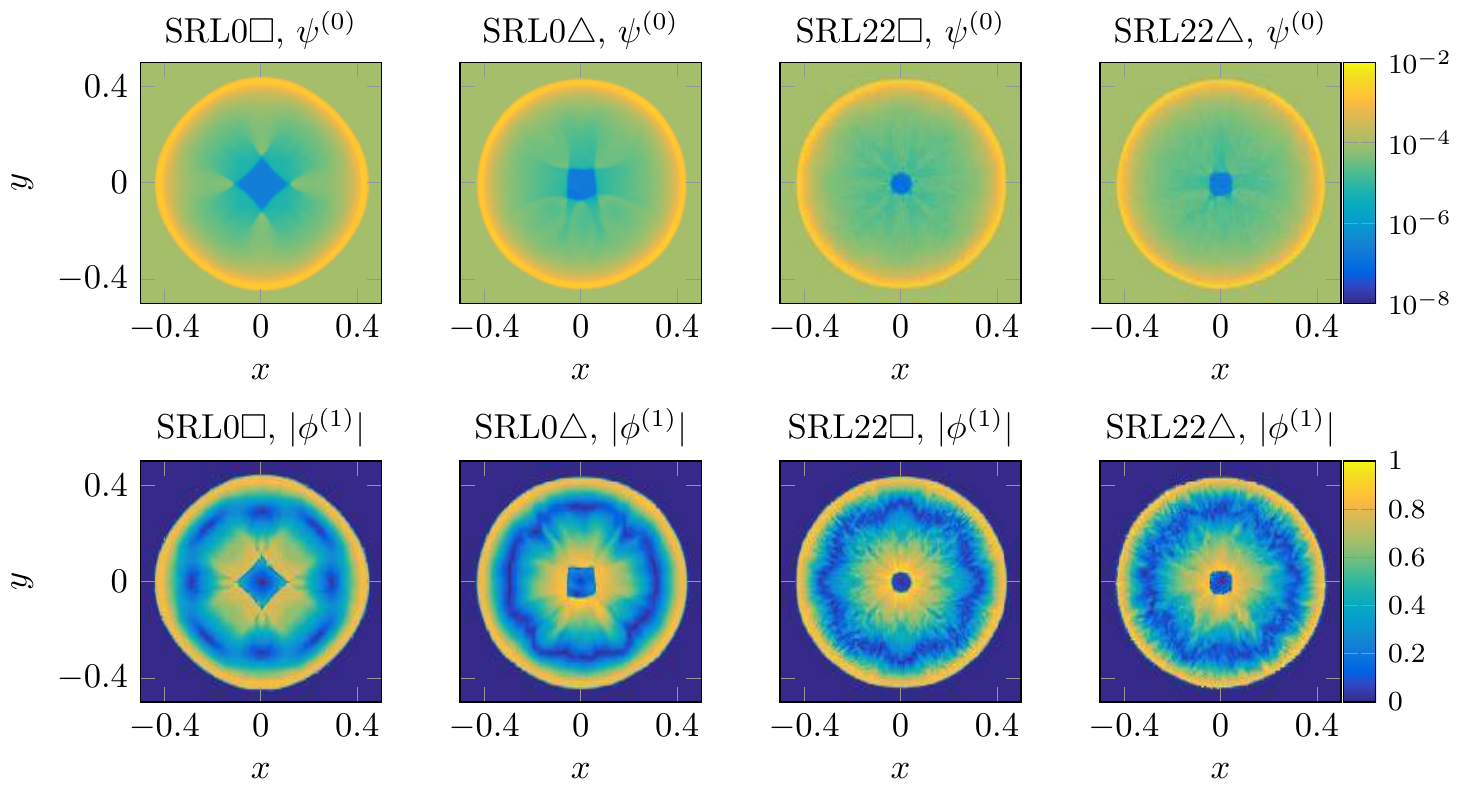}{\input{Images/Linesource4c}}
\caption{Line source: primitive variable and realizability limiter.}
\label{fig:linesource4c}
\end{figure}

As we have shown, using the characteristic and realizability limiter together yields good results. However, the change to characteristic fields is expensive and the transformation matrices become arbitrarily poorly conditioned close to the realizability boundary \cite{AllHau12}, yielding inaccurate results. Unfortunately, Figure \ref{fig:linesource4c} shows that the combination of limiting in the 
primitive-variable and realizability limiter is not enough to obtain good results. It appears that these solutions are even worse than the realizability-limiter-only solutions shown in Figure~\ref{fig:linesource2}.

    Our numerical tests show that in general the characteristic limiter combined with the realizability limiter (CRLM) is superior
    to the combination of limiting in the components and the realizability limiter (SRLM). However, the differences between the two limiting
    strategies are most dramatic for the linesouce test case. In the rest of the problems we consider, the solutions both look good
    and are comparable.

\begin{table}
\centering
\begin{tabular}{l|| r@{.}l r@{.}l r@{.}l r@{.}l r@{.}l }
& \multicolumn{2}{c}{SL$\infty\square$}& \multicolumn{2}{c}{SL$0\square$}& \multicolumn{2}{c}{SL$22\square$} & \multicolumn{2}{c}{CL$0\square$}& \multicolumn{2}{c}{CL$22\square$}\\
\cmidrule(r){1-1}\cmidrule(r){2-3} \cmidrule(r){4-5} \cmidrule(r){6-7} \cmidrule(r){8-9} \cmidrule(r){10-11} \
$\max\limits_t$ GP & 11 & 49\%& 8 & 19\%& 13 & 01\%& 0 & 84\%& 7 & 69\%\\
$\max\limits_t$ CM & 10 & 09\%& 8 & 08\%& 10 & 59\%& 0 & 81\%& 5 & 87\%\\
GP$(t_{f})$ & 11 & 49\%& 3 & 26\%& 13 & 01\%& 0 & 06\%& 7 & 69\%\\
CM$(t_{f})$ & 10 & 09\%& 3 & 15\%& 10 & 59\%& 0 & 00\%& 5 & 87\%\\
\end{tabular}
\caption{Line source: percentage of non-realizable cells and quadrature points on rectangular mesh.}
\label{tab:linesource_rect_table}
\end{table}
\begin{table}
\centering
\begin{tabular}{l|| r@{.}l r@{.}l r@{.}l r@{.}l r@{.}l }
& \multicolumn{2}{c}{SL$\infty\triangle$}& \multicolumn{2}{c}{SL$0\triangle$}& \multicolumn{2}{c}{SL$22\triangle$} & \multicolumn{2}{c}{CL$0\triangle$}& \multicolumn{2}{c}{CL$22\triangle$}\\
\cmidrule(r){1-1}\cmidrule(r){2-3} \cmidrule(r){4-5} \cmidrule(r){6-7} \cmidrule(r){8-9} \cmidrule(r){10-11} \
$\max\limits_t$ GP & 81 & 01\%& 5 & 01\%& 30 & 10\%& 0 & 03\%& 0 & 33\%\\
$\max\limits_t$ CM & 80 & 33\% & 5 & 03\%& 25 & 12\%& 0 & 00\%& 0 & 17\%\\
GP$(t_{f})$ & 81 & 01\%& 5 & 05\%& 30 & 10\% & 0& 01\%& 0& 20\%\\
CM$(t_{f})$ & 80 & 33\% & 5 & 00\%& 25 & 12\%& 0 & 00\%& 0 & 08\%\\
\end{tabular}
\caption{Line source: percentage of non-realizable cells and quadrature points on triangular mesh.}
\label{tab:linesource_tri_table}
\end{table}

In Tables~\ref{tab:linesource_rect_table} and \ref{tab:linesource_tri_table} we summarize the performance of the slope limiter in the absence of the realizability 
limiter by counting the proportion of non-realizable cells (CM) and Gauss points (GP). We provide both the maximum values in the time interval (taken over a sample of $20$ equidistant
time points) and the values at the final time step for each mesh. It is clear that the characteristic limiter (CL$M$$\square/\triangle$) performs better than the slope limiter in primitive variables
(SL$M$$\square/\triangle$). In fact the final solution from the characteristic limiter (CL$0$$\square/\triangle$) is realizable in the mean on both meshes. However, realizability in the mean and at the Gauss points is not guaranteed
at all time steps on the time interval. These results further reinforce the need for the realizability limiter.

%%%%%%%%%%%%%%%%%%%%%%%%%%%%%%%%%%%%%%%%%%
%Hdisk
%\newpage
\subsection{Homogeneous Disk}
Our next test case is the two-dimensional version of the homogeneous sphere test, which is often used to test radiation transport codes, see e.g. \cite{Abdikamalov2012,Bruenn1985,Rampp2002,Sumiyoshi2012}. It consists of a static homogeneous and isothermal sphere that radiates in vacuum \cite{Radice2013648}. It is well-suited to test moment approximations since it admits an analytical solution \cite{smit1997hyperbolicity}. The parameters in the two-dimensional setting are as follows:
\begin{itemize}
\item Domain: $[-5,5]^2$ with $h=0.05$
\item Time: 3.00
\item Parameters: Let $D$ be the unit disk $D=\{ (x,y):\ x^2+y^2\leq 1\}$. Set $\sigma_s=0$; $\sigma_a=10$ on $D$, zero otherwise; source $q^{(0)}=1$ on $D$, zero otherwise; $q^{(1)} = 0$
\item Initial condition: $\psi^{(0)}=10^{-10}$, $\psi^{(1)}=0$
\item Boundary conditions: $\psi^{(0)}=10^{-10}$, $\psi^{(1)}=0$
\end{itemize}

This test case is numerically challenging due to the discontinuous parameters. Furthermore, its solution is again radially symmetric, simplifying the analysis of the quality of the limiter configurations.

\begin{figure}[h]
\def\localpath{./Images/HomDisk/}
\centering
\externaltikz{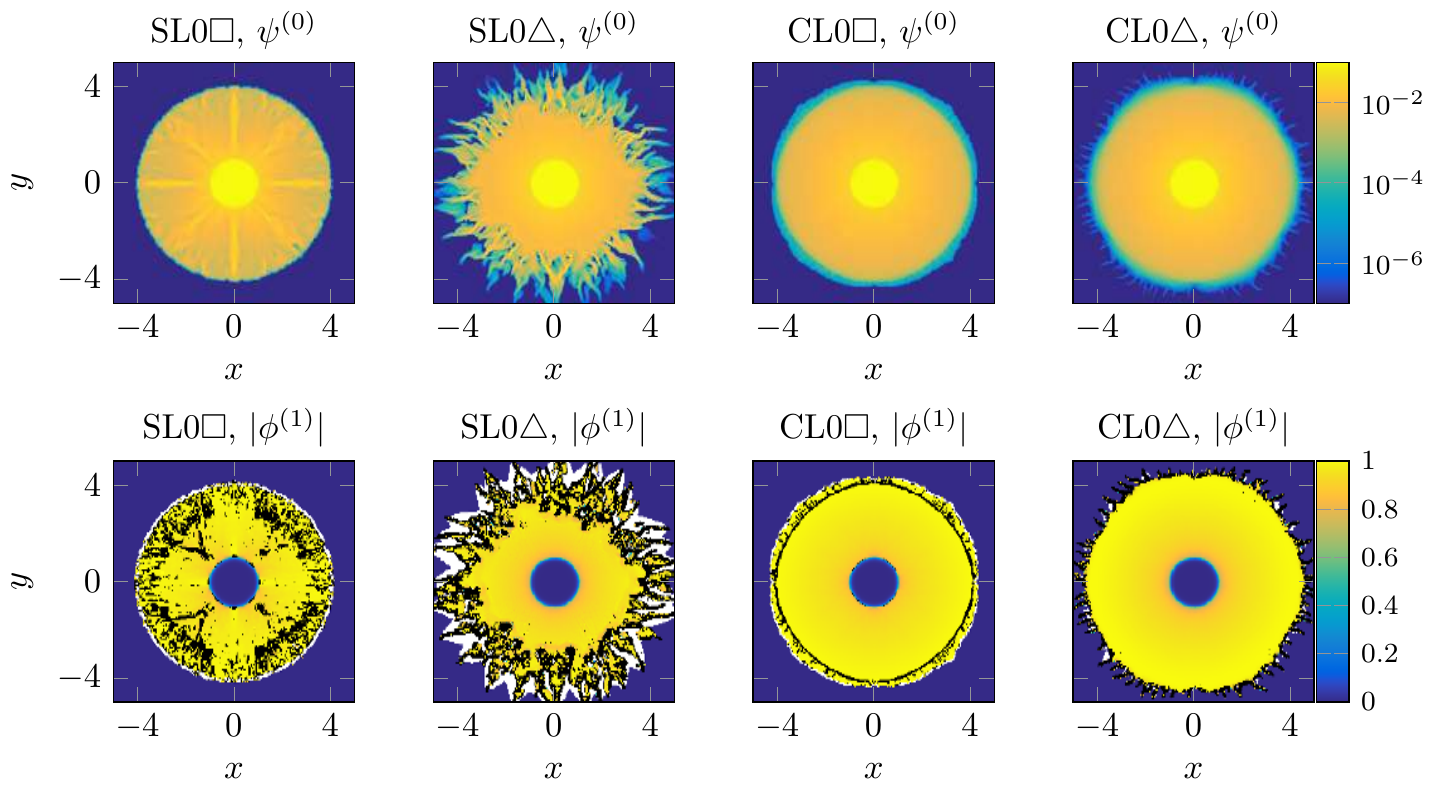}{\input{Images/Hdisk4a}}
\caption{Homogeneous disk: comparison of primitive and characteristic slope limiter.}
\label{fig:Hdiska}
\end{figure}
As in the line source test case, we observe that the application of the primitive-variable limiter results in an oscillatory solution which is not realizable. This is shown in Figure~\ref{fig:Hdiska}. In case of the triangular mesh strong filaments form at the edge of the disk. Using the characteristic limiter slightly improves the solution. However, and in contrast to the line source, even with $M=0$ 
the characteristic limiter is no longer able to preserve realizability in the mean (see Tables~\ref{tab:hdisk_rect_table} and~\ref{tab:hdisk_tri_table}), furthermore, a larger proportion of
the Gauss points are not realizable. Additionally, at the boundary of the disk the quality of the solution is strongly reduced due to the bad condition number of the transformation matrix, which 
goes to infinity for values of $\lvert\phi^{1}\rvert$ approaching $1$ \cite{AllHau12}.

\begin{figure}[h]
\def\localpath{./Images/HomDisk/}
\centering
\externaltikz{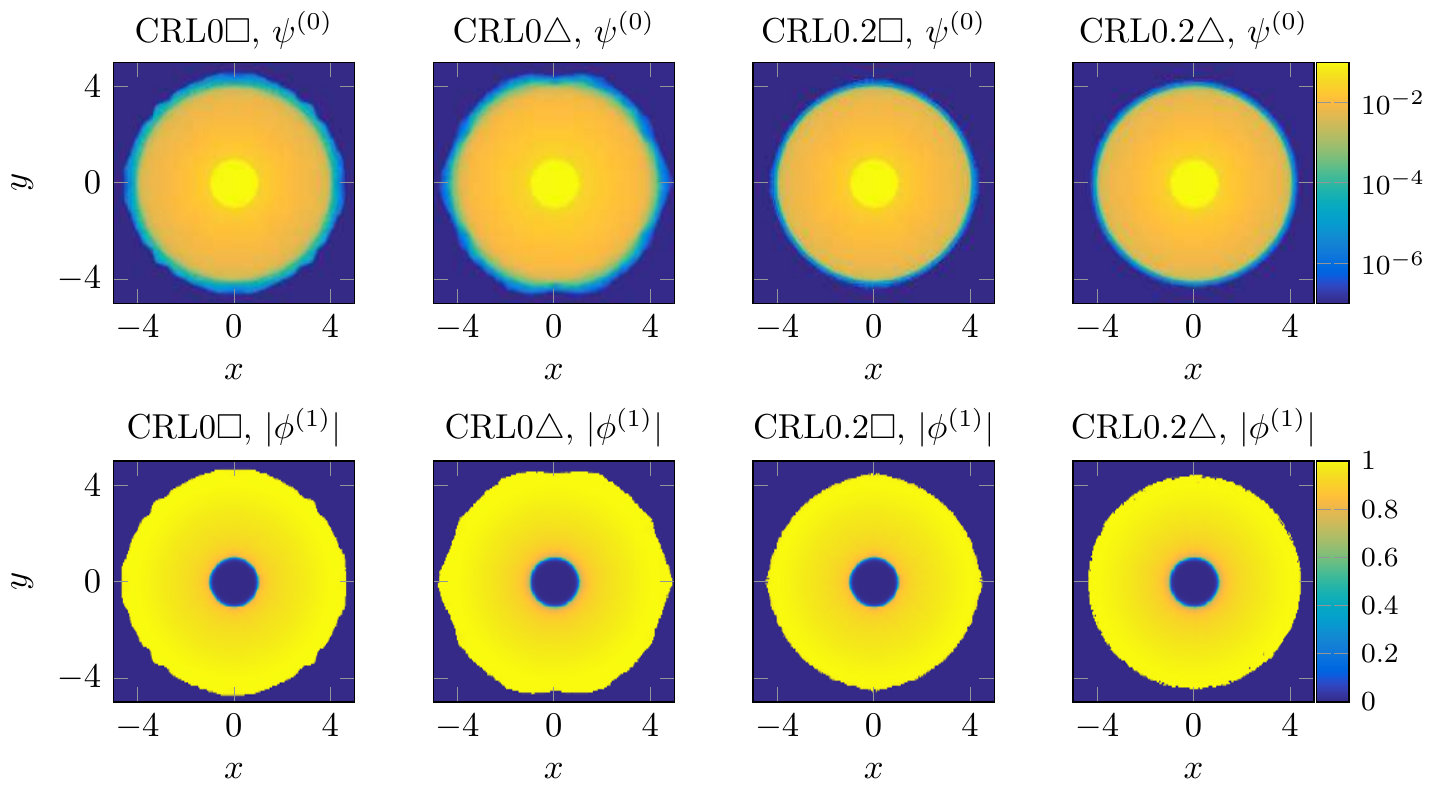}{\input{Images/Hdisk4b}}
\caption{Homogeneous disk: comparison of the characteristic limiter with and without realizability limiter}
\label{fig:Hdiskb}
\end{figure}
Adding the realizability limiter strongly improves the quality of the solution (Figure~\ref{fig:Hdiskb}). As before, choosing $M=0.2$ gives less diffusive results.

\begin{table}
\centering
\begin{tabular}{l|| r@{.}l r@{.}l r@{.}l r@{.}l r@{.}l }
& \multicolumn{2}{c}{SL$\infty\square$}& \multicolumn{2}{c}{SL$0\square$}& \multicolumn{2}{c}{SL$0.2\square$} & \multicolumn{2}{c}{CL$0\square$}& \multicolumn{2}{c}{CL$0.2\square$}\\
\cmidrule(r){1-1}\cmidrule(r){2-3} \cmidrule(r){4-5} \cmidrule(r){6-7} \cmidrule(r){8-9} \cmidrule(r){10-11} \
$\max\limits_t$ GP  & 24 & 13\%& 12 & 46\%& 18 & 95\%& 1 & 57\%& 5 & 53\%\\
$\max\limits_t$ CM  & 21 & 95\%& 7 & 19\%& 15 & 08\%& 0 & 22\%& 1 & 14\%\\
GP$(t_{f})$ & 4& 93\%&12& 46\%& 18 & 46\%& 1 & 23\%& 2 & 91\%\\
CM$(t_{f})$ & 0& 26\%& 7 & 19\%& 9 & 39\%& 0 & 06\%& 0 & 06\%\\
\end{tabular}
\caption{Homogeneous disk: percentage of non-realizable cells and quadrature points on rectangular mesh.}
\label{tab:hdisk_rect_table}
\end{table}
\begin{table}
\centering
\begin{tabular}{l|| r@{.}l r@{.}l r@{.}l r@{.}l r@{.}l }
& \multicolumn{2}{c}{SL$\infty\triangle$}& \multicolumn{2}{c}{SL$0\triangle$}& \multicolumn{2}{c}{SL$0.2\triangle$} & \multicolumn{2}{c}{CL$0\triangle$}& \multicolumn{2}{c}{CL$0.2\triangle$}\\
\cmidrule(r){1-1}\cmidrule(r){2-3} \cmidrule(r){4-5} \cmidrule(r){6-7} \cmidrule(r){8-9} \cmidrule(r){10-11} \
$\max\limits_t$ GP & 71 & 75\%& 34&25\% &  41 & 73\%& 5 & 65\% & 33 & 65\%\\
$\max\limits_t$ CM & 71 & 62\% & 34&41\% & 40 & 18\%& 5 & 75\% & 32 & 64\%\\
GP$(t_{f})$ & 71 & 75\%& 34 & 25\%& 41 & 73\%& 5 & 65\%& 33 & 65\%\\
CM$(t_{f})$ & 71 & 61\%& 34 & 41\%& 40 & 18\%& 5 & 71\%& 32 & 64\%\\
\end{tabular}
\caption{Homogeneous disk: percentage of non-realizable cells and quadrature points on triangular mesh.}
\label{tab:hdisk_tri_table}
\end{table}
A quantitative analysis of the slope limiters (SL$M$$\square/\triangle$) and (CL$M$$\square/\triangle$) show that limiting in characteristic variables
yields better results but the solution remains unrealizable.
%\newpage
\subsection{Flash}
%\comment{M=0.5}

In this test case a bulk of mass is moving from the center of the domain to the right boundary \cite{Science2013}. The parameters are given as follows:
\begin{itemize}
\item Domain: $[-10,10]^2$ with $h =0.06$
\item Time: 6
\item Parameters: Let $D$ be unit disk with radius $0.5$, $D=\{ (x,y):\ x^2+y^2\leq 0.25\}$. Set $\sigma_s=0$; $\sigma_a=0$; source $q^{(0)}=0$, $q^{(1)} = 0$
\item Initial condition: $\psi^{(0)}=1$, $\psi^{(1)}_x=0.9$, $\psi^{(1)}_y=0$ on $D$, $\psi^{(0)}=10^{-10}$, $\psi^{(1)}=0$ otherwise
\item Boundary conditions: $\psi^{(0)}=10^{-10}$, $\psi^{(1)}=0$
\end{itemize}

Due to the initial condition which is placed close to the boundary of realizability this test is well suited to show the inability of standard slope limiters to preserve realizability. This is demonstrated in Figure~\ref{fig:flash4a} where similar effects as before occur.

\begin{figure}[htbp]
\def\localpath{./Images/Flash/}
\centering
\externaltikz{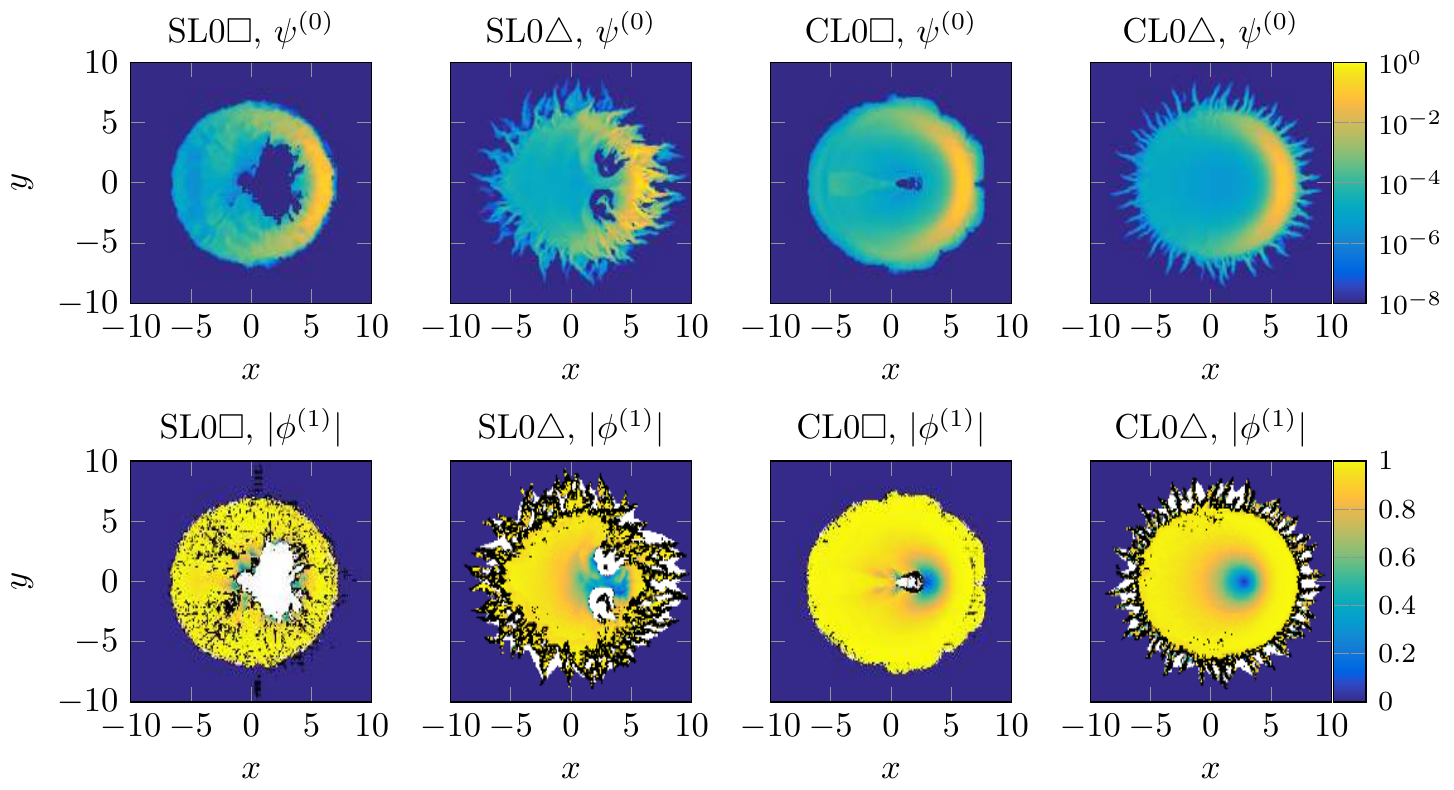}{\input{Images/flash4a}}
\caption{Flash test: comparison of primitive and characteristic slope limiter.}
\label{fig:flash4a}
\end{figure}

\begin{figure}[htbp]
\def\localpath{./Images/Flash/}
\centering
\externaltikz{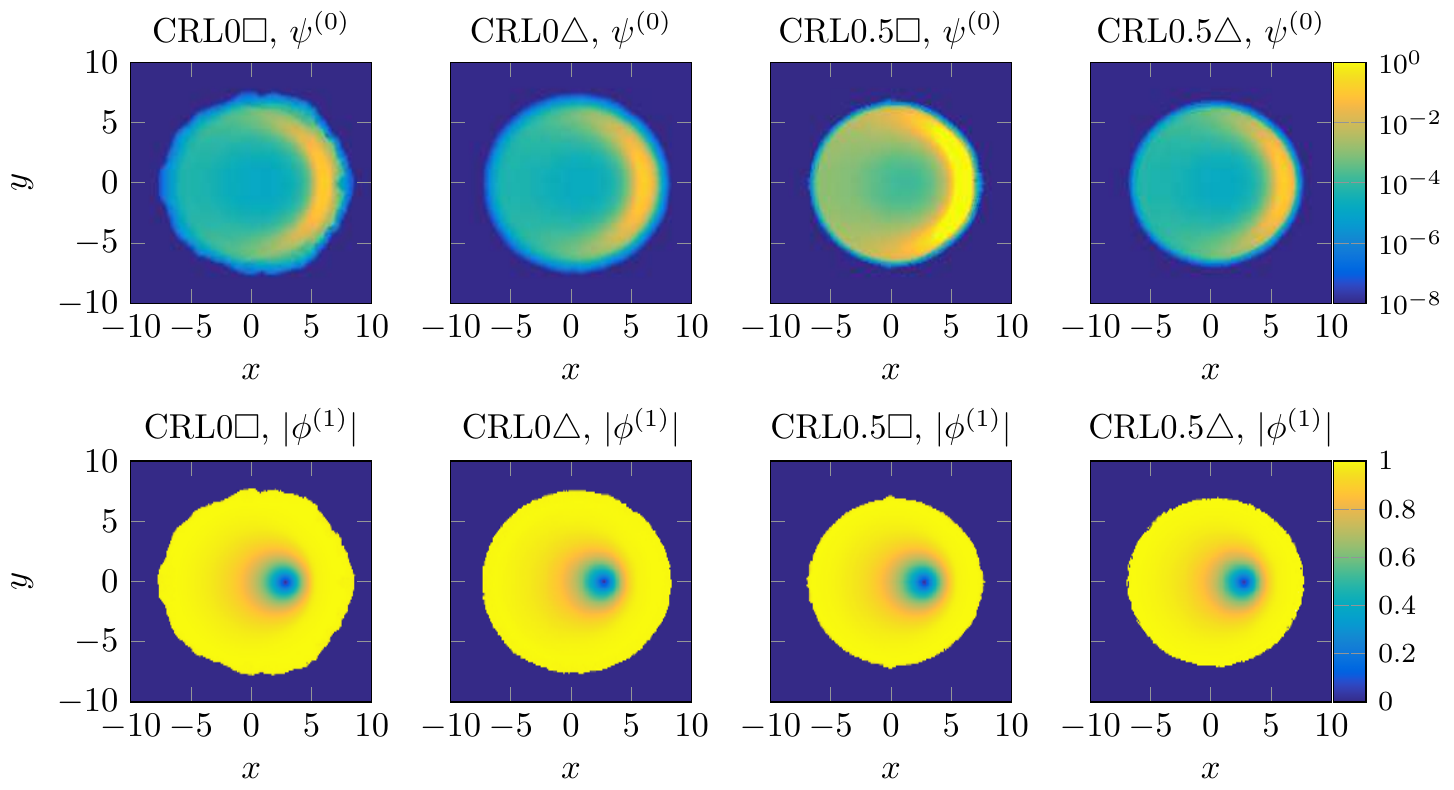}{\input{Images/flash4}}
\caption{Flash test: characteristic limiter with realizability limiter for different values of $M$.}
\label{fig:flash4}
\end{figure}

A comparison of Figures~\ref{fig:flash4a} and~\ref{fig:flash4} confirms that the combination of the characteristic limiter with the realizability
limiter yields the best solutions. The performance of the slope limiter without the realizability limiter in Tables~\ref{tab:flash_rect_table} and~\ref{tab:flash_tri_table}
is consistent with previous observations in the line source and homogeneous disk test cases.

\begin{table}
\centering
\begin{tabular}{l|| r@{.}l r@{.}l r@{.}l r@{.}l r@{.}l }
& \multicolumn{2}{c}{SL$\infty\square$}& \multicolumn{2}{c}{SL$0\square$}& \multicolumn{2}{c}{SL$0.5\square$} & \multicolumn{2}{c}{CL$0\square$}& \multicolumn{2}{c}{CL$0.5\square$}\\
\cmidrule(r){1-1}\cmidrule(r){2-3} \cmidrule(r){4-5} \cmidrule(r){6-7} \cmidrule(r){8-9} \cmidrule(r){10-11} \
$\max\limits_t$ GP & 14 & 84\%& 13 & 97\%& 20 & 88\%& 0 & 86\%& 19 & 75\%\\
$\max\limits_t$ CM & 11 & 77\%& 11 & 53\%& 14 & 27\%& 0 & 22\%& 15 & 01\%\\
GP$(t_{f})$ & 14 & 70\%& 13 & 97\%& 20 & 88\%& 0 & 86\%& 19 & 75\%\\
CM$(t_{f})$ & 9 & 50\%& 11 & 53\%& 14 & 27\%& 0 & 22\%& 15 & 01\%\\
\end{tabular}
\caption{Flash: percentage of non-realizable cells and quadrature points on rectangular mesh}
\label{tab:flash_rect_table}
\end{table}

\begin{table}
\centering
\begin{tabular}{l|| r@{.}l r@{.}l r@{.}l r@{.}l r@{.}l }
& \multicolumn{2}{c}{SL$\infty\triangle$}& \multicolumn{2}{c}{SL$0\triangle$}& \multicolumn{2}{c}{SL$0.5\triangle$} & \multicolumn{2}{c}{CL$0\triangle$}& \multicolumn{2}{c}{CL$0.5\triangle$}\\
\cmidrule(r){1-1}\cmidrule(r){2-3} \cmidrule(r){4-5} \cmidrule(r){6-7} \cmidrule(r){8-9} \cmidrule(r){10-11} \
$\max\limits_t$ GP & 55 & 21\%& 24 & 12\%& 50 & 14\%& 5 & 45\%& 50 & 99\%\\
$\max\limits_t$ CM & 55 & 18\% & 24 & 27\% & 49 & 10\% & 5 & 63\%& 39 & 00\%\\
GP$(t_{f})$ & 55 & 28\%& 24 & 12\%& 50 & 14\%& 5 & 45\%& 40 & 99\%\\
CM$(t_{f})$ & 55 & 18\%& 24 & 27\%& 49 & 10\%& 5 & 63\%& 39 & 00\%\\
\end{tabular}
\caption{Flash: percentage of non-realizable cells and quadrature points on triangular mesh.}
\label{tab:flash_tri_table}
\end{table}
\section{Conclusions and outlook}
In this work, we have investigated a third order realizability-preserving DG method to approximate solutions to the $M_1$ model of radiation transport. The results show that in all test cases presented, the realizability limiter is imperative to ensure the realizability of the numerical solutions.

Based on these results, we advocate the $M_1$ model as an ideal test case that pushes the boundaries of high-order realizability-preserving numerical methods. The model equations are highly nonlinear, and are well-posed only in a geometrically complex (yet convex) domain of realizability. Furthermore, practically relevant solutions are not just wave fronts, but rather beam-like solutions, or they possess Dirac-like source terms. Solutions and material coefficients tend to vary over several orders of magnitude. Furthermore, so-called source-detector problems in radiation transport require an accurate solution (in terms of the relative error) at one or more specific points in space.

The results presented in this paper cast some doubt on whether DG is a competitive numerical scheme for $M_1$, considering the accuracy requirements commonly encountered in practice. For instance, for the homogeneous disk test case the DG solution possesses around $10^6$ degrees of freedom. It is comparable in quality to a Lax-Friedrichs (LF) solution with roughly the same number of degrees of freedom ($2048\times 2048\approx 4\cdot 10^6$), which however is significantly faster to compute because no limiting is needed. Similarly, in the line source test case, the best DG solution is comparable to a $1024\times 1024$ LF solution.

In several test cases, the high-order DG method yields reliable results only if a full limiting strategy is employed, including transformations to and from the characteristic variables. Not only do those transformations incur a significant cost, they also become less and less well conditioned when higher accuracy requirements are imposed. It is therefore questionable whether, for the $M_1$ model, the high-order DG methodology will ever become more efficient than a simple monotone first-order scheme that automatically guarantees realizability.

That being said, it should be pointed out that the DG methodology studied here, albeit being an established standard method, can most certainly be improved. For instance, more sophisticated limiting techniques, such as WENO limiting with KXRCF shock detection \cite{Krivodonova2004,Qiu2005a}, or limiters that incorporate high-order approximations during the limiting process~\cite{Burbeau2001,Krivodonova2007}, could be combined with the realizability-preserving limiter. This would remove the need for the problem-dependent minmod parameter $M$. Furthermore, more efficient time stepping schemes could be applied, such as explicit low-storage SSP Runge-Kutta schemes \cite{Ketcheson2008}, SSP multi-step Runge-Kutta methods \cite{Ketcheson,Bresten2013} (which both require less evaluations of the differential operator $\mathcal{L}_h$), or even Implicit-Explicit schemes \cite{Ascher1997,Kanevsky2007} (to remove the dependence on the right-hand side ($\sigma_a$ and $\sigma_s$) in the CFL condition \eqref{eqn:cfl}).

Further research steps include realizability limiting for higher-order minimum-entropy models (i.e., $M_N$ with $N\geq 2$) in two and three dimensions, similar to \cite{Schneider2015a}. The main challenge in that context is that a complete characterization of the realizability set by means of simple algebraic criteria is still outstanding. One possible remedy is to use discrete realizability criteria with respect to a quadrature rule \cite{AllHau12}. Moreover, in general a high number of quadrature points is necessary to produce solutions without discernible discrete ray artifacts \cite{Schneider2016,Schneider2015a,Garrett2014}, resulting in an extremely high cost for realizability limiting. Hence less expensive strategies to guarantee realizability need to be developed.

\section*{Acknowledgements}
B.\ Seibold wishes to acknowledge support by the National Science
Foundation under grant DMS-1719640. M.\ Frank acknowledges support from Deutsche Forschungsgemeinschaft (DFG) under Grant FR 2841/6-1.
We thank Philipp Monreal for co-developing the unstructured-mesh DG code.

% Bibliography
%%%%%%%%%%%%%%
\bibliographystyle{bibstyle}
\bibliography{RadLit}

\begin{thebibliography}{10}

\bibitem{Abdikamalov2012}
E.~Abdikamalov, A.~Burrows, C.~D. Ott, F.~L{\"{o}}ffler, E.~O'Connor, J.~C.
  Dolence, and E.~Schnetter.
\newblock {A New Monte Carlo Method for Time-Dependent Neutrino Radiation
  Transport}.
\newblock \emph{Astrophys. J.}, 755(2):page 111, 2012.

\bibitem{AllHau12}
G.~Alldredge, C.~Hauck, and A.~Tits.
\newblock High-order entropy-based closures for linear transport in slab
  geometry {II}: A computational study of the optimization problem.
\newblock \emph{SIAM. J. Sci. Comput}, 34(4): pp.  B361--B391, 2012.

\bibitem{Schneider2015a}
G.~W. Alldredge and F.~Schneider.
\newblock {A realizability-preserving discontinuous Galerkin scheme for
  entropy-based moment closures for linear kinetic equations in one space
  dimension}.
\newblock \emph{J. Comput. Phys.}, 295: pp.  665--684, 2015.

\bibitem{Ascher1997}
U.~Ascher, S.~Ruuth, and R.~Spiteri.
\newblock {Implicit--explicit Runge-Kutta methods for time-dependent partial
  differential equations}.
\newblock \emph{Applied Numerical Mathematics}, 1997.

\bibitem{Berthon2006}
C.~Berthon, P.~Charrier, and B.~Dubroca.
\newblock An {HLLC} scheme to solve the ${M}_1$ model of radiative transfer in
  two space dimensions.
\newblock \emph{J. Sci. Comput.}, 31(3): pp.  347--389, 2006.

\bibitem{biswas1994parallel}
R.~Biswas, K.~D. Devine, and J.~E. Flaherty.
\newblock Parallel, adaptive finite element methods for conservation laws.
\newblock \emph{Appl. Numer. Math.}, 14(1): pp.  255--283, 1994.

\bibitem{Bresten2013}
C.~Bresten, S.~Gottlieb, and Z.~Grant.
\newblock {Strong Stability Preserving Multistep Runge-Kutta Methods}.
\newblock \emph{arXiv preprint}, 2013.

\bibitem{Bruenn1985}
S.~W. Bruenn.
\newblock {Stellar core collapse - Numerical model and infall epoch}.
\newblock \emph{Astrophys. J. Suppl. Ser.}, 58:page 771, 1985.

\bibitem{BruHol01}
T.~A. Brunner and J.~P. Holloway.
\newblock One-dimensional {R}iemann solvers and the maximum entropy closure.
\newblock \emph{J. Quant. Spectrosc. Radiat. Transfer}, 69: pp.  543--566,
  2001.

\bibitem{Brunner2005b}
---.
\newblock {Two-dimensional time dependent Riemann solvers for neutron
  transport}.
\newblock \emph{J. Comput. Phys}, 210(September 2001): pp.  386--399, 2005.

\bibitem{Burbeau2001}
A.~Burbeau, P.~Sagaut, and C.-H. Bruneau.
\newblock A problem-independent limiter for high-order runge–kutta
  discontinuous galerkin methods.
\newblock \emph{Journal of Computational Physics}, 169(1): pp.  111 -- 150,
  2001.
\newblock ISSN 0021-9991.

\bibitem{Chidyagwai2016}
P.~Chidyagwai, F.~Schneider, and P.~Monreal.
\newblock {Realizability-preserving Runge-Kutta Discontinuous-Galerkin schemes
  for the M1 model of radiative transfer: Code}, 2016.

\bibitem{Cockburn}
B.~Cockburn.
\newblock {An introduction to the discontinuous Galerkin method for
  convection-dominated problems}.
\newblock \emph{Advanced numerical approximation of nonlinear {\ldots}}.

\bibitem{CockburnShuIV}
B.~Cockburn, S.~Hou, and C.-W. Shu.
\newblock {TVB} {R}unge-{K}utta local projection discontinuous {G}alerkin
  finite element method for conservation laws {IV}: {T}he multi-dimensional
  case.
\newblock \emph{Math. Comp.}, 54: pp.  545--581, 1990.

\bibitem{CockburnLinShu}
B.~Cockburn, S.-Y. Lin, and C.-W. Shu.
\newblock {TVB} {R}unge-{K}utta local projection discontinuous {G}alerkin
  finite element method for consevation laws {III}: {O}ne-dimensional systems.
\newblock \emph{J. Comput. Phys.}, 84: pp.  90 -- 113, 1989.

\bibitem{CockburnShuPk}
B.~Cockburn and C.-W. Shu.
\newblock {TVB} {R}unge-{K}utta local projection discontinuous {G}alerkin
  finite element method for conservation laws {II}: {G}eneral framework.
\newblock \emph{Math. Comp.}, 52(186): pp.  411--435, 1989.

\bibitem{CockburnShuP1}
B.~Cockburn and C.~W. Shu.
\newblock The {R}unge-{K}utta local projection ${P}^1$-discontinuous-{G}alerkin
  finite element method for scalar conservation laws.
\newblock \emph{ESAIM. Math. Model. Num.}, 25(3): pp.  337--361, 1991.

\bibitem{CockburnShuV}
B.~Cockburn and C.-W. Shu.
\newblock The {R}unge--{K}utta discontinuous {G}alerkin method for conservation
  laws {V}: {M}ultidimensional systems.
\newblock \emph{J. Comput. Phys.}, 141(2): pp.  199--224, 1998.

\bibitem{review_article}
---.
\newblock Runge--kutta discontinuous galerkin methods for convection-dominated
  problems.
\newblock \emph{J. Sci.Comput.}, 16(3): pp.  173--261, 2001.

\bibitem{Goudon2005}
J.-F. Coulombel, F.~Golse, and T.~Goudon.
\newblock {Diffusion approximation and entropy-based moment closure for kinetic
  equations}.
\newblock \emph{Asymptotic. Anal}, 45(1,2): pp.  1--39, 2005.

\bibitem{Dre87}
W.~Dreyer.
\newblock Maximisation of the entropy in non-equilibrium.
\newblock \emph{J. Phys. A}, 20: pp.  6505--6517, 1987.

\bibitem{DubFeu99}
B.~Dubroca and J.~L. Feugeas.
\newblock Entropic hierarchy for the radiative transfer equation.
\newblock \emph{C. R. Acad. Sci. Paris Ser. I}, 329: pp.  915--920, 1999.

\bibitem{Frank07}
M.~Frank, H.~Hensel, and A.~Klar.
\newblock A fast and accurate moment method for the {F}okker--{P}lanck equation
  and applications to electron radiotherapy.
\newblock \emph{SIAM. J. Appl. Math.}, 67(2): pp.  582--603, 2007.

\bibitem{Garrett2014}
C.~K. Garrett and C.~D. Hauck.
\newblock {A Comparison of Moment Closures for Linear Kinetic Transport
  Equations: The Line Source Benchmark}.
\newblock \emph{Transp. Theory. Stat. Phys.}, 2013.

\bibitem{gottlieb_shu_tadmor}
S.~Gottlieb, C.-W. Shu, and E.~Tadmor.
\newblock Strong stability-preserving high-order time discretization methods.
\newblock \emph{SIAM Rev.}, 43(1): pp.  89--112, 2001.

\bibitem{Gough_GSL}
B.~Gough.
\newblock GNU Scientific Library Reference Manual - Third Edition.
\newblock Network Theory Ltd., 3rd edition, 2009.
\newblock ISBN 0954612078, 9780954612078.

\bibitem{hesthaven2007nodal}
J.~S. Hesthaven and T.~Warburton.
\newblock Nodal discontinuous Galerkin methods: algorithms, analysis, and
  applications.
\newblock Springer Science \& Business Media, 2007.

\bibitem{jaynes1957info}
E.~T. Jaynes.
\newblock Information theory and statistical mechanics.
\newblock \emph{Phys. Rev.}, 106(4):page 620, 1957.

\bibitem{Kanevsky2007}
A.~Kanevsky, M.~H. Carpenter, D.~Gottlieb, and J.~S. Hesthaven.
\newblock {Application of implicit–explicit high order Runge–Kutta methods
  to discontinuous-Galerkin schemes}.
\newblock \emph{Journal of Computational Physics}, 225(2): pp.  1753--1781,
  2007.
\newblock ISSN 00219991.

\bibitem{Ker76}
D.~Kershaw.
\newblock Flux limiting nature's own way: A new method for numerical solution
  of the transport equation.
\newblock Technical report, LLNL Report UCRL-78378, 1976.

\bibitem{Ketcheson2008}
D.~I. Ketcheson.
\newblock {Highly Efficient Strong Stability-Preserving Runge–Kutta Methods
  with Low-Storage Implementations}.
\newblock \emph{SIAM Journal on Scientific Computing}, 30(4): pp.  2113--2136,
  2008.
\newblock ISSN 1064-8275.

\bibitem{Ketcheson}
D.~I. Ketcheson, S.~Gottlieb, and C.~B. Macdonald.
\newblock {Strong stability preserving two-step Runge-Kutta methods}.
\newblock \emph{SIAM Journal on Numerical {\ldots}},  pp.  1--24, 2011.

\bibitem{Krivodonova2007}
L.~Krivodonova.
\newblock Limiters for high-order discontinuous galerkin methods.
\newblock \emph{Journal of Computational Physics}, 226(1): pp.  879 -- 896,
  2007.
\newblock ISSN 0021-9991.

\bibitem{Krivodonova2004}
L.~Krivodonova, J.~Xin, J.-F. Remacle, N.~Chevaugeon, and J.~Flaherty.
\newblock {Shock detection and limiting with discontinuous Galerkin methods for
  hyperbolic conservation laws}.
\newblock \emph{Appl. Numer. Math.}, 48(3-4): pp.  323--338, 2004.

\bibitem{Larsen1991}
E.~W. Larsen and G.~C. Pomraning.
\newblock {The PN Theory as an Asymptotic Limit of Transport Theory in Planar
  Geometry —I: Analysis}.
\newblock \emph{Nuclear Science and Engineering}, 109(1): pp.  49--75, 1991.

\bibitem{LeVeque2002}
R.~LeVeque.
\newblock {Finite volume methods for hyperbolic problems}.
\newblock 2002.

\bibitem{Lev84}
C.~D. Levermore.
\newblock Relating {E}ddington factors to flux limiters.
\newblock \emph{J. Quant. Spectrosc. Radiat. Transfer}, 31: pp.  149--160,
  1984.

\bibitem{levermore1996moment}
---.
\newblock Moment closure hierarchies for kinetic theories.
\newblock \emph{J. Stat. Phys.}, 83(5--6): pp.  1021--1065, 1996.

\bibitem{levermore2009boundary}
---.
\newblock {Boundary conditions for moment closures}.
\newblock \emph{Institute for Pure and Applied Mathematics University of
  California, Los Angeles, CA on May}, 27, 2009.

\bibitem{li2011central}
F.~Li, L.~Xu, and S.~Yakovlev.
\newblock {C}entral discontinuous {G}alerkin methods for ideal {MHD} equations
  with the exactly divergence-free magnetic field.
\newblock \emph{J. Comput. Phys.}, 230(12): pp.  4828--4847, 2011.

\bibitem{liangdiscontinuous}
C.~Liang, F.~Ham, and E.~Johnsen.
\newblock {D}iscontinuous {G}alerkin method with {WENO} limiter for flows with
  discontinuity.
\newblock Technical report, Center for Turbulence Research, 2009.

\bibitem{liu_xu_gas_kinetic}
H.~Liu, , and K.~Xu.
\newblock A gas-kinetic discontinuous {G}alerkin method for viscous flow
  equations.
\newblock \emph{J. Mech. Sci. Technol.}, 21(9): pp.  1344--1351, 2007.

\bibitem{liu2007runge}
H.~Liu and K.~Xu.
\newblock {A} {R}unge-{K}utta discontinuous {G}alerkin method for viscous flow
  equations.
\newblock \emph{J. Comput. Phys.}, 224(2): pp.  1223--1242, 2007.

\bibitem{Monreal}
P.~Monreal.
\newblock Moment Realizability and Kershaw Closures in Radiative Transfer.
\newblock Ph.D. thesis, TU Aachen, 2012.

\bibitem{MullerRuggeri}
I.~M{\"u}ller and T.~Ruggeri.
\newblock Rational Extended Thermodynamics.
\newblock Springer-Verlag, New York, 2nd edition, 1993.

\bibitem{Olbrant12}
E.~Olbrant, C.~D. Hauck, and M.~Frank.
\newblock A realizability-preserving discontinuous galerkin method for the
  ${M}_1$ model of radiative transfer.
\newblock \emph{J. Comput. Phys.}, 231(17): pp.  5612 -- 5639, 2012.

\bibitem{Ore55}
A.~Ore.
\newblock Entropy of radiation.
\newblock \emph{Phys. Rev.}, 98:page 887, 1955.

\bibitem{pomraning1964variational}
G.~C. Pomraning.
\newblock {Variational boundary conditions for the spherical harmonics
  approximation to the neutron transport equation}.
\newblock \emph{Annals of Physics}, 27(2): pp.  193--215, 1964.

\bibitem{Qiu2005a}
J.~Qiu and C.~Shu.
\newblock {A comparison of troubled cell indicators for Runge-Kutta
  discontinuous Galerkin methods using WENO Limiters}.
\newblock \emph{SIAM J. Sci. Comput.}, 27: pp.  995--1013, 2005.

\bibitem{Radice2013648}
D.~Radice, E.~Abdikamalov, L.~Rezzolla, and C.~D. Ott.
\newblock {A new spherical harmonics scheme for multi-dimensional radiation
  transport I. Static matter configurations}.
\newblock \emph{J. Comput. Phys.}, 242(0): pp.  648--669, 2013.

\bibitem{Rampp2002}
M.~Rampp and H.~T. Janka.
\newblock {Radiation hydrodynamics with neutrinos: Variable Eddington factor
  method for core-collapse supernova simulations}.
\newblock \emph{Astron. Astrophys.}, 396:page~42, 2002.

\bibitem{reedhill}
W.~Reed and T.~Hill.
\newblock Triangular mesh methods for neutron transport equation.
\newblock \emph{Los Alamos Scientific Laboratory Report}, LA-UR-73-479, 1973.

\bibitem{Ros54}
P.~Rosen.
\newblock Entropy of radiation.
\newblock \emph{Phys. Rev.}, 96:page 555, 1954.

\bibitem{Rulko1991}
R.~P. Rulko, E.~W. Larsen, and G.~C. Pomraning.
\newblock {The PN Theory as an Asymptotic Limit of Transport Theory in Planar
  Geometry —II: Numerical Results}.
\newblock \emph{Nuclear Science and Engineering}, 109(1): pp.  76--85, 1991.

\bibitem{Schneider2016}
F.~Schneider.
\newblock {Moment models in radiation transport equations}.
\newblock Dr. Hut Verlag, 2016.

\bibitem{Science2013}
F.~Science, Y.~Kanno, T.~Harada, and T.~Hanawa.
\newblock {Kinetic Scheme for Solving the M1 Model of Radiative Transfer}.
\newblock \emph{Publ. Astron. Soc. Jpn.}, 65(2007):page~72, 2013.

\bibitem{CWShu}
C.~Shu.
\newblock {TVD} time discretizations.
\newblock \emph{SIAM J. Sci. Stat. Comput.}, 9: pp.  1073--1084, 1988.

\bibitem{Shu1998}
C.-W. Shu.
\newblock {Essentially non-oscillatory and weighted essentially non-oscillatory
  schemes for hyperbolic conservation laws}.
\newblock 1998.

\bibitem{Shu1988}
C.-W. Shu and S.~Osher.
\newblock Efficient implementation of essentially non-oscillatory
  shock-capturing schemes.
\newblock \emph{J. Comput. Phys.}, 77(2): pp.  439--471, 1988.

\bibitem{smit1997hyperbolicity}
J.~M. Smit, J.~Cernohorsky, and C.~P. Dullemond.
\newblock {Hyperbolicity and critical points in two-moment approximate
  radiative transfer.}
\newblock \emph{Astron. Astrophys.}, 325: pp.  203--211, 1997.

\bibitem{Struchtrup2000}
H.~Struchtrup.
\newblock {Kinetic schemes and boundary conditions for moment equations}.
\newblock \emph{Zeitschrift f{\"{u}}r angewandte Mathematik und Physik},
  51(3):page 346, 2000.
\newblock ISSN 00442275.
\newblock \doi{10.1007/s000330050002}.

\bibitem{Sumiyoshi2012}
K.~Sumiyoshi and S.~Yamada.
\newblock {Neutrino Transfer in Three Dimensions for Core-Collapse Supernovae.
  I. Static Configurations}.
\newblock \emph{Astrophys. J. Suppl. Ser.}, 199(1):page~17, 2012.

\bibitem{toro2009riemann}
E.~Toro.
\newblock Riemann Solvers and Numerical Methods for Fluid Dynamics: A Practical
  Introduction.
\newblock Springer, 2009.
\newblock ISBN 9783540252023.

\bibitem{Wang2012}
C.~Wang, X.~Zhang, C.~W. Shu, and J.~Ning.
\newblock {Robust high order discontinuous Galerkin schemes for two-dimensional
  gaseous detonations}.
\newblock \emph{J. Comput. Phys.}, 231(2): pp.  653--665, 2012.

\bibitem{Xing2013}
Y.~Xing and X.~Zhang.
\newblock {Positivity-preserving well-balanced discontinuous Galerkin methods
  for the shallow water equations on unstructured triangular meshes}.
\newblock \emph{J. Sci. Comput.},  pp.  1--35, 2013.

\bibitem{Zhang2010}
X.~Zhang and C.~W. Shu.
\newblock {On positivity-preserving high order discontinuous Galerkin schemes
  for compressible Euler equations on rectangular meshes}.
\newblock \emph{J. Comput. Phys.}, 229(23): pp.  8918--8934, 2010.

\bibitem{Zhang2012a}
X.~Zhang and C.-W. Shu.
\newblock {A minimum entropy principle of high order schemes for gas dynamics
  equations}.
\newblock \emph{Numer. Math.}, 121(3): pp.  545--563, 2012.

\bibitem{ZhangShu12}
X.~Zhang, Y.~Xia, and C.-W. Shu.
\newblock Maximum-principle-satisfying and positivity-preserving high order
  discontinuous {G}alerkin schemes for conservation laws on triangular meshes.
\newblock \emph{J. Sci. Comput.}, 50(1): pp.  29--62, 2012.

\end{thebibliography}

\end{document}